\documentclass[10pt]{article} 
\usepackage[utf8]{inputenc} 
\usepackage[english]{babel}
\usepackage[T1]{fontenc}
\usepackage{latexsym}

\usepackage{geometry} 
\usepackage{booktabs} 
\usepackage{array} 
\usepackage{paralist} 
\usepackage{verbatim} 
\usepackage{subfig} 
\usepackage{amsmath}
\usepackage{amsfonts}
\usepackage{hyperref}
\usepackage{bbm}
\usepackage{babel}
\usepackage{amsthm}
\usepackage{amssymb}
\usepackage{mathtools}
\usepackage{bbold}

\makeatletter
\newtheorem*{rep@theorem}{\rep@title}
\newcommand{\newreptheorem}[2]{%
\newenvironment{rep#1}[1]{%
 \def\rep@title{#2 \ref{##1}}%
 \begin{rep@theorem}}%
 {\end{rep@theorem}}}
\makeatother

\newtheorem{thm}{Theorem}[section]
\newtheorem{cor}[thm]{Corollary}
\newtheorem{lem}[thm]{Lemma}
\newtheorem{prop}[thm]{Proposition}
\newtheorem{defn}[thm]{Definition}
\newtheorem{rem}[thm]{Remark}
\newtheorem{esem}[thm]{Example}

\newtheorem*{cor*}{Corollary}

\newtheorem{theorem}{Theorem}
\newreptheorem{theorem}{Theorem}

\newreptheorem{proposition}{Proposition}

\newreptheorem{corollary}{Corollary}

\usepackage{fancyhdr} 
\title{Continuity method with movable singularities for classical complex Monge-Ampère equations.}
\author{Antonio Trusiani\footnote{email: \href{mailto:antonio.trusiani91@gmail.com}{antonio.trusiani91@gmail.com}}}
\date{} 

\begin{document}
\maketitle
\begin{abstract}
On a compact Kähler manifold $(X,\omega)$, we study the strong continuity of solutions with prescribed singularities of complex Monge-Ampère equations with integrable Lebesgue densities. Moreover, we give sufficient conditions for the strong continuity of solutions when the right-hand sides are modified to include all (log) Kähler-Einstein metrics with prescribed singularities. Our findings can be interpreted as closedness of new continuity methods in which the densities vary together with the prescribed singularities. For Monge-Ampère equations of Fano type, we also prove an openness result when the singularities decrease. As an application, we deduce a strong stability result for (log-)Kähler Einstein metrics on semi-Kähler classes given as modifications of $\{\omega\}$.
\end{abstract}
\vspace{5pt}
{\small \textbf{Keywords:} Complex Monge-Ampère equations, compact Kähler manifolds, Kähler-Einstein metrics.}
\section{Introduction.}
Let $(X,\omega)$ be a compact Kähler manifold endowed with a Kähler form. This article concerns the study of (degenerate) complex Monge-Ampère equations of type
\begin{equation}
\label{eqn:MAFirst}
\begin{cases}
MA_{\omega}(u)=e^{-\lambda u}f\omega^{n}\\
u\in PSH(X,\omega)
\end{cases}
\end{equation}
where $PSH(X,\omega)$ denotes the set of all $\omega$-plurisubharmonic functions on $X$, $MA_{\omega}(u)=(\omega+dd^{c}u)^{n}$ in the sense of the non-pluripolar product (\cite{BEGZ10}), $\lambda \in \mathbbm{R}$ and $f\in L^{1}\setminus \{0\}$. Here $dd^{c}=\frac{i}{2\pi}\partial\bar{\partial}$ as $d^{c}:=\frac{i}{4\pi}(\bar{\partial} -\partial)$.\\
The analysis of these equations plays a main role in several questions in Kähler geometry, such as in the search of (log) Kähler-Einstein metrics (\cite{Yau78}, \cite{Tian00}). A very classical tool is the \emph{continuity method} in which the density $g_{1}(u):=e^{-\lambda u}f$ is approximated by a family $\{g_{t}\}_{t\in[0,1]}$ so that, to solve (\ref{eqn:MAFirst}), it is enough to check that the set of all $t\in[0,1]$ such that $MA_{\omega}(u)=g_{t}(u)\omega^{n}$ admits a solution is not-empty, open and closed. Usually, the closedness is the most involved task, which is related to the regularity of solutions and to the kind of convergence requested.\newline

In this paper we examine the closedness of a new continuity method with \emph{movable singularities}, i.e. we allow the solutions to have some prescribed singularities and we require a \emph{strong convergence}.\\
More precisely, for $\psi\in PSH(X,\omega)$ we denote by $\mathcal{E}(X,\omega,\psi)$ the set of all $\omega$-psh functions with $\psi$-relative full Monge-Ampère mass, i.e. all $u\in PSH(X,\omega), u\leq \psi + C$ for a constant $C\in\mathbbm{R}$ such that $V_{u}:=\int_{X}MA_{\omega}(u)=\int_{X}MA_{\omega}(\psi)=:V_{\psi}$, while $\mathcal{E}^{1}(X,\omega,\psi)\subset \mathcal{E}(X,\omega,\psi)$ is composed by all functions with finite $\psi$-relative Monge-Ampère energy $E_{\psi}(\cdot)$ (\cite{DDNL17b}, \cite{BEGZ10}, see section \S \ref{sec:Pre}). Roughly speaking, there sets include all functions that are slightly more singular than $\psi$, and in particular all $u\in PSH(X,\omega)$ such that $u-\psi$ is globally bounded belong to $\mathcal{E}^{1}(X,\omega,\psi)$.\newline
Then for $\lambda\in\mathbbm{R}$, letting $\{f_{k}\}_{k\in\mathbbm{N}}$ be a sequence of non-negative functions $L^{1}$-converging to $f$, we assume to have a family of solutions $\{u_{k}\}_{k\in\mathbbm{N}}$ of
\begin{equation}
\label{eqn:Maeqnk}
\begin{cases}
MA_{\omega}(u_{k})=e^{-\lambda u_{k}}f_{k}\omega^{n}\\
u\in\mathcal{E}^{1}(X,\omega,\psi_{k})
\end{cases}
\end{equation}
and we give sufficient (and sometimes necessary) conditions for a strong convergence of $u_{k}$ to a solution $u$ of
\begin{equation}
\label{eqn:Maeqn}
\begin{cases}
MA_{\omega}(u)=e^{-\lambda u}f\omega^{n}\\
u\in\mathcal{E}^{1}(X,\omega,\psi).
\end{cases}
\end{equation}
Here $\psi_{k},\psi\in PSH(X,\omega)$ clearly represent the \emph{prescribed singularities} while seeking solutions in $\mathcal{E}^{1}$ may be thought as a regularity constraint.

As proved in \cite{DDNL17b} there is a natural asumption to add on the prescribed singularities $\psi$ to make a complex Monge-Ampère equation $MA_{\omega}(u)=\mu$, for $\mu$ normalized positive non-pluripolar measure, always solvable in the class $\mathcal{E}(X,\omega,\psi)$: $\psi$ must be a \emph{model type envelope} (see again section \S \ref{sec:Pre}). We denote by $\mathcal{M}^{+}$ the set of all model type envelopes $\psi$ such that $V_{\psi}>0$.\newline

The most interesting case to analyze is when the singularities are increasing or decreasing, so we suppose to have a totally ordered sequence $\{\psi_{k}\}_{k\in\mathbbm{N}}\subset \mathcal{M}^{+}$ converging weakly (i.e. in the usual $L^{1}$-topology) to an element $\psi\in \mathcal{M}^{+}$, recalling that the natural partial order $\preccurlyeq$ on $PSH(X,\omega)$ is given by $u\leq v+C$ for a constant $C\in\mathbbm{R}$.

We can then ask for the \emph{strong convergence} of solutions in the sense of \cite{Tru19} and \cite{Tru20}, i.e. $u_{k}\to u$ strongly if $u_{k}\to u$ weakly and $E_{\psi_{k}}(u_{k})\to E_{\psi}(u)$. In fact this kind of convergence is very natural, implies the convergence in capacity (\cite[Theorem 6.3]{Tru20}) and it is equivalent to a metric convergence with respect to a complete distance $d_{\mathcal{A}}$ on $X_{\mathcal{A}}:=\bigsqcup_{\psi\in\overline{\mathcal{A}}}\mathcal{E}^{1}(X,\omega,\psi)$  for $\mathcal{A}:=\{\psi_{k}\}_{k\in\mathbbm{N}}$ (\cite{Tru19}).\\

To state the results, we need to distinguish three different cases based on the sign of $\lambda$.

When $\lambda=0$ the uniqueness of solutions holds modulo translation by constant, while (\ref{eqn:Maeqn}) is solvable if and only if $f\omega^{n}\in \mathcal{M}^{1}(X,\omega,\psi)$ by \cite[Theorem A]{Tru20}. We refer to \S 2.2 for the definition of $\mathcal{M}^{1}(X,\omega,\psi)$, but we underline here that this set contains all measures $f\omega^{n}$ such that $\int_{X}f\omega^{n}=V_{\psi}$ (necessary condition) and such that $f\in L^{p}$ for $p>1$.
\begin{theorem}
\label{thmA}
Assume
\begin{itemize}
\item[(i)] $f_{k},f\in L^{1}\setminus\{0\}$ non-negative such that $f_{k}\to f$ in $L^{1}$;
\item[(ii)] $\{\psi_{k}\}_{k\in \mathbbm{N}}\subset\mathcal{M}^{+}$ totally ordered such that $\psi_{k}\to \psi\in \mathcal{M}^{+}$ weakly;
\item[(iii)] $f_{k}\omega^{n}\in \mathcal{M}^{1}(X,\omega,\psi_{k})$ for any $k\in \mathbbm{N}$, and denote by $u_{k}\in\mathcal{E}^{1}_{norm}(X,\omega,\psi_{k})$ the unique solution with $\sup_{X}u_{k}=0$ of (\ref{eqn:Maeqnk}) for $\lambda=0$.
\end{itemize}
Let $u$ be a weak accumulation point of $\{u_{k}\}_{k\in \mathbbm{N}}$. Then $u\in\mathcal{E}_{norm}(X,\omega,\psi)$ and it satisfies $MA_{\omega}(u)=f\omega^{n}$. Furthermore, $u\in \mathcal{E}^{1}_{norm}(X,\omega,\psi)$ and $u_{k}\to u$ strongly if and only if $E_{\psi_{k}}(u_{k})\geq -C$ for a uniform constant $C\geq 0$ and
\begin{equation}
\label{eqn:WeirdH}
\limsup_{k\to \infty}\int_{X}(\psi_{k}-u_{k})f_{k}\omega^{n}\leq \int_{X}(\psi-u)f\omega^{n}.
\end{equation}
\end{theorem}
With obvious notations, $\mathcal{E}^{1}_{norm}(X,\omega,\psi):=\{u\in\mathcal{E}^{1}(X,\omega,\psi)\,: \, \sup_{X}u=0\}$ and similarly for $\mathcal{E}_{norm}(X,\omega,\psi)$. The existence of a weak accumulation point $u$ is a consequence of standard $L^{1}$-compactness arguments.\newline
As the condition (\ref{eqn:WeirdH}) and the uniform bound on the energies may be difficult to detect, in Remark \ref{rem:SomeCases} we collect some easier cases in which these conditions are fulfilled. We stress that the main difficulty is that $f\in L^{1}$ but a priori $f\notin L^{p}$ for any $p>1$. Indeed, in this general case, it is usually a complex task to check if the unique $u\in \mathcal{E}_{norm}(X,\omega,\psi)$ satisfying $MA_{\omega}(u)=f\omega^{n}$ belongs to $\mathcal{E}^{1}_{norm}(X,\omega,\psi)$, which is a regularity condition, and Theorem \ref{thmB} provides an effective new tool (to prove $u\in \mathcal{E}^{1}(X,\omega,\psi)$ the condition (\ref{eqn:WeirdH}) is unnecessary).\\

If $\lambda<0$ then (\ref{eqn:Maeqn}) admits a unique solution by \cite[Theorem 4.23]{DDNL17b} as the assumption of the small unbounded locus becomes unnecessary thanks to \cite{X19a}. In this case there are no obstruction to the strong convergence as our next Theorem shows.
\begin{theorem}
\label{thmB}
Assume
\begin{itemize}
\item[(i)] $\lambda<0$;
\item[(ii)] $f_{k},f\in L^{1}\setminus\{0\}$ non-negative such that $f_{k}\to f$ in $L^{1}$;
\item[(iii)] $\{\psi_{k}\}_{k\in \mathbbm{N}}\subset\mathcal{M}^{+}$ totally ordered such that $\psi_{k}\to \psi\in\mathcal{M}^{+}$ weakly.
\end{itemize}
Let $u_{k}\in\mathcal{E}^{1}(X,\omega,\psi_{k})$, $u\in\mathcal{E}^{1}(X,\omega,\psi)$ be the unique solutions respectively of
\begin{equation*}
\begin{cases}
MA_{\omega}(u_{k})=e^{-\lambda u_{k}}f_{k}\omega^{n}\\
u_{k}\in\mathcal{E}^{1}(X,\omega,\psi_{k}),
\end{cases}
\begin{cases}
MA_{\omega}(u)=e^{-\lambda u}f\omega^{n}\\
u\in\mathcal{E}^{1}(X,\omega,\psi).
\end{cases}
\end{equation*}
Then $u_{k}\to u$ strongly.
\end{theorem}

Finally, the case $\lambda>0 $ is much more complicated. For instance, if $\psi=0$, $\{\omega\}=-K_{X}$ and $f=0$ then any solution of (\ref{eqn:Maeqn}) corresponds to a Kähler-Einstein metric on a Fano manifold, whose existence is characterized by a purely algebro-geometric criterion (see \cite{CDS15}) while the uniqueness depends on the identity component of the automorphism group $\mbox{Aut}(X)^{\circ}$ (\cite{BM86}).\newline
However, we develop a variational approach for a functional $F_{f,\psi,\lambda}$, which generalizes the Ding functional (\cite{Ding88}), and we impose an integrability condition in terms of the classical \emph{complex singularity exponent} (see for instance \cite{DK99}) to show the following openness result of the continuity method (see also Corollary \ref{cor:Open}).
\begin{theorem}
\label{thmC}
Let $\psi\in\mathcal{M}^{+}$, $\lambda>0$ and $f\in L^{p}$ for $p\in (1,\infty]$. Assume also that $c(\psi)>\frac{\lambda p}{p-1}$ where $\frac{\lambda p}{p-1}=\lambda$ if $p=\infty$. If the functional $F_{f,\psi,\lambda}$ is coercive then there exists a constant $A>1$ such that the complex Monge-Ampère equation
\begin{equation*}
\begin{cases}
MA_{\omega}(u)=e^{-\lambda u}f\omega^{n}\\
u\in\mathcal{E}^{1}(X,\omega,\psi')
\end{cases}
\end{equation*}
admits a solution for any $\psi'\in \mathcal{M}^{+}$, $\psi'\succcurlyeq \psi$ that satisfies $V_{\psi'}< AV_{\psi}$.
\end{theorem}
In Theorem \ref{thmC}, the \emph{coercivity} of $F_{f,\psi,\lambda}$ is with respect to the $\psi$-relative $J$-functional (or, equivalently, with respect to the distance $d_{\mathcal{A}|\mathcal{E}^{1}_{norm}(X,\omega,\psi)\times \mathcal{E}^{1}_{norm}(X,\omega,\psi)}$ associated to the strong convergence). \newline
Concerning the integrability assumption, note also that the higher is $p$, the more singular $\psi$ can be. In the limit case $p=+\infty$ the condition $c(\psi)>\lambda$ is necessary to solve the Monge-Ampère equation.\newline
A key point in proving Theorem \ref{thmC} is to relate the coercivity of $F_{f,\psi,\lambda}$ to a Moser-Trudinger type of inequality passing through the coercivity of a functional that is reminiscent of the Mabuchi functional (see Proposition \ref{prop:MT}).\newline

About the strong continuity of solutions in the case $\lambda>0$, i.e. the closedness of the continuity method, we prove the following result.
\begin{theorem}
\label{thmD}
Let $\lambda>0$, $\{\psi_{k}\}_{k\in \mathbbm{N}}\subset \mathcal{M}^{+}$ be a totally ordered sequence converging to $\psi\in\mathcal{M}^{+}$, and $f_{k},f\geq 0$ non trivial such that $f_{k}\to f$ in $L^{p}$ as $k\to \infty$ for $p\in (1,\infty]$. Assume also the following conditions:
\begin{itemize}
\item[(i)] $c(\psi)>\frac{\lambda p}{p-1}$;
\item[(ii)] the complex Monge-Ampère equations 
$$
\begin{cases}
MA_{\omega}(u_{k})=e^{-\lambda u_{k}}f_{k}\omega^{n}\\
u_{k}\in \mathcal{E}^{1}(X,\omega,\psi_{k});
\end{cases}
$$
admit solutions $u_{k}$ given as maximizers of $F_{f_{k},\psi_{k},\lambda}$;
\item[(iii)] $\sup_{X} u_{k}\leq C$ for a uniform constant $C$.
\end{itemize}
Then there exists a subsequence $\{u_{k_{h}}\}_{h\in\mathbbm{N}}$ that converges strongly to $u\in\mathcal{E}^{1}(X,\omega,\psi)$ solution of
$$
\begin{cases}
MA_{\omega}(u)=e^{-\lambda u}f\omega^{n}\\
u\in\mathcal{E}^{1}(X,\omega,\psi).
\end{cases}
$$
\end{theorem}
Like in many settings we expect that any solution of (\ref{eqn:Maeqnk}) maximizes $F_{f,\psi,\lambda}$ (see for instance \cite{Tru20b} for the search of Kähler-Einstein metrics on Fano manifolds) and the assumption $(i)$ is satisfied for many $\psi\in\mathcal{M}^{+}$, the unique real obstacle is the uniform estimate in $(iii)$ as in other continuity methods. In fact, this assumption is necessary even if we force the singularities to move non trivially and we consider the easiest case $f_{k}\equiv f\in L^{\infty}$ as Example \ref{esem:Necess} shows.\\

In the second part of the paper, we apply our results to the study of special metrics. More precisely, given $\omega$ Kähler form, $\psi\in\mathcal{M}^{+}$, $D$ $\mathbbm{Q}$-divisor, we say that $\omega+dd^{c}u$ is a \emph{$(D,[\psi])$-log KE metric} if $u\in\mathcal{E}^{1}(X,\omega,\psi)$ and
$$
Ric(\omega+dd^{c}u)-[D]=\lambda (\omega+dd^{c}u)
$$
in a weak sense for $\lambda\in\mathbbm{Q}$ (see section \S \ref{sec:KE}). This abuse of language is justified by the fact that $\omega+dd^{c}u$ actually defines a (class of) singular log KE metric. The definition of log KE metrics easily extends to $\mathbbm{R}$-divisors $D$, $\lambda\in\mathbbm{R}$ and to \emph{semi-Kähler} forms, i.e. when $\omega$ is smooth, semipositive and satisfies $\int_{X}\omega^{n}>0$.\\

Then we introduce $\mathcal{M}^{+}_{an}$ as the set of all model type envelopes $\psi$ with \emph{analytic singularities types}, i.e. all $\psi$ such that $\psi-\varphi$ is globally bounded for a $\omega$-psh function $\varphi$ with \emph{analytic singularities}. Using the log-resolutions of the associated ideal sheaves, we denote by $\mathcal{K}_{(X,\omega)}$ the image of the map
\begin{equation}
\label{eqn:Map}
\Phi:\mathcal{M}^{+}_{an}\to\big\{(Y,\eta)\,:\,  \eta \,\mbox{semi-Kähler with}\, \omega\geq p_{*}\eta,\, \mbox{and} \, p:Y\to X \, \mbox{given by a sequence of blow-ups}\big\}/\sim
\end{equation}
where $(Y,\eta)\sim(Y',\eta')$ if there exists another element $(Z,\tilde{\eta})$ which dominates $(Y,\eta)$, $(Y',\eta')$ in the usual way. $\mathcal{K}_{(X,\omega)}$ inherits a partial order (we say \emph{smaller}, \emph{bigger} in correspondence of $\preccurlyeq$, $\succcurlyeq$), and it is possible to define a \emph{log-KE metric in} $\alpha\in\mathcal{K}_{(X,\omega)}$ as a class of log-KE metrics on any representative $(Y,\eta)$ of $\alpha$. Moreover, when $\{\alpha_{k}\}_{k\in\mathbbm{N}}$ is a totally ordered sequence, there is a natural \emph{strong convergence} for a sequence of log-KE metrics in $\{\alpha_{k}\}_{k\in\mathbbm{N}}$ which is clearly deduced from the strong convergence on $PSH(X,\omega)$. In particular, when $\alpha_{k},\alpha$ have representatives on the same compact Kähler manifold $Y$, this strong convergence implies the weak convergence of the log-KE metrics on $Y$.
\begin{theorem}
\label{thmE}
Let $\omega$ be a Kähler form and $D$ be a klt $\mathbbm{R}$-divisor such that $c_{1}(X)-\{[D]\}=\lambda\{\omega\}$ for $\lambda\in\mathbbm{R}$. Then any $(D,\psi)$-log KE metric $\omega+dd^{c}u$ for $\psi\in\mathcal{M}_{an}^{+}$ satisfies $u\in\mathcal{C}^{\infty}(X\setminus A)$ for a closed analytic set $A$. Furthermore, the following statements hold.
\begin{itemize}
\item[(i)] Suppose $\lambda\leq 0$. Then there is a unique log-KE metric for any element $\mathcal{K}_{(X,\omega)}$ and these metrics are stable with respect to the strong convergence, i.e. if $\{\alpha_{k}\}_{k\in\mathbbm{N}}\subset \mathcal{K}_{(X,\omega)}$ is a totally ordered sequence converging to $\alpha\in\mathcal{K}_{(X,\omega)}$ then the associated log-KE metrics strongly converge. 
\item[(ii)] Suppose $\lambda>0$ and let $\alpha\in\mathcal{K}_{(X,\omega)}$. If the log-Ding functional associated to $(Y,\eta)$, representative of $\alpha$, is coercive, then any $\alpha'\in \mathcal{K}_{(X,\omega)}$ slightly bigger than $\alpha$ admits a log-KE metric.
\item[(iii)] Suppose $\lambda>0$. If $\{\alpha_{k}\}_{k\in\mathbbm{N}}\subset \mathcal{K}_{(X,\omega)}$ is an increasing sequence converging to $\alpha\in \mathcal{K}_{(X,\omega)}$ that is \emph{uniformly bounded from above}, then it admits a subsequence that converges strongly to a log-KE metric in $\alpha$.
\end{itemize}
\end{theorem}
Some comments about Theorem \ref{thmE}.\\
The topological assumption $c_{1}(X)-\{[D]\}=\lambda \{\omega\}$ is clearly necessary while the klt hypothesis becomes necessary when $\lambda\geq 0$. In particular, there are no obstruction to the existence and the strong convergence in the case $\lambda=0$, while we do not investigate the case $\lambda<0$ with $(X,D)$ not necessarily klt as it is beyond the purpose of this paper. The last two points are clearly consequences of Theorems \ref{thmC}, \ref{thmD}, but it is worth to underline that there are not assumptions on the class $\alpha$ corresponding to the integrability condition $c(\psi)>\frac{\lambda p}{p-1}$. The uniform boundedness from above in $(iii)$ corresponds to the last (necessary) assumption of Theorem \ref{thmD}.
\subsection{Structure of the paper.}
After recalling some preliminaries in section \S \ref{sec:Pre}, section \S \ref{sec:Degenerate} is the core of the paper where in three different subsections based on the sign of $\lambda$ we prove Theorems \ref{thmA}, \ref{thmB}, \ref{thmC} and \ref{thmD}. Section \S \ref{sec:KE} then concerns the definition of the $(D,[\psi])$-log KE metrics and the proof of Theorem \ref{thmE}. 
\subsection{Acknowledgments.}
I would like to thank my PhD advisors Stefano Trapani and David Witt Nyström for their numerous comments. The author is supported by a postdoctoral grant of the Knut and Alice Wallenberg
Foundation.
\section{Preliminaries.}
\label{sec:Pre}
The set of all model type envelopes is defined as
$$
\mathcal{M}:=\{\psi\in PSH(X,\omega)\, : \, \psi=P_{\omega}[\psi](0)\}.
$$
where for any couple of $\omega$-psh functions $u,v$
$$
P_{\omega}[u](v):=\Big(\lim_{C\to \infty}P_{\omega}(u+C,v)\Big)^{*}=\Big(\sup\{w\in PSH(X,\omega)\, : \, w\preccurlyeq u, w\leq v\}\Big)^{*}\in PSH(X,\omega).
$$
The star is for the upper semicontinuous regularization and $P_{\omega}(u,v):=\big(\sup\{w\in PSH(X,\omega)\, :\, w\leq \min(u,v)\}\big)^{*}$ (\cite{RWN14}). We set $P_{\omega}[\psi]:=P_{\omega}[\psi](0)$ for simplicity and we denote by $\mathcal{M}^{+}$ the elements $\psi\in\mathcal{M}$ such that $V_{\psi}:=\int_{X}MA_{\omega}(\psi)>0$ (\cite{Tru19}). \newline
There are plenty of elements in $\mathcal{M}$ since $P_{\omega}[P_{\omega}[\psi]]=P_{\omega}[\psi]$ if $V_{\psi}>0$, i.e. $v\to P_{\omega}[v]$ may be thought as a projection from the set of $\omega$-psh functions to $\mathcal{M}$. Moreover the preimage of this projection operator at $\psi\in \mathcal{M}^{+}$ includes $\mathcal{E}(X,\omega,\psi):=\{u\in PSH(X,\omega)\, : \, u\preccurlyeq \psi, V_{u}=V_{\psi}\}$, i.e. $P_{\omega}[u]=\psi$ for any $u\in PSH(X,\omega)$ with $\psi$-relative full Monge-Ampère mass (\cite[Theorem 1.3]{DDNL17b}). Observe also that $\sup_{X}u=\sup_{X}(u-\psi)$ if $u\preccurlyeq \psi$, $\psi\in\mathcal{M}$ as an immediate consequence of $u-\sup_{X}u\leq \psi\leq 0$.\newline

As stated in the Introduction, the set of model type envelopes is crucial to study complex Monge-Ampère equations. Indeed, $\psi-P_{\omega}[\psi]$ bounded is necessary to make the equation
$$
\begin{cases}
MA_{\omega}(u)=\mu\\
u\in\mathcal{E}(X,\omega,\psi)
\end{cases}
$$
always solvable where $\mu$ is a non-pluripolar measure such that $\mu(X)=V_{\psi}$ (\cite[Theorem 4.34]{DDNL17b}). Thus, without loss of generality we may assume $\psi$ be a model type envelope. \newline
It is also necessary for the sequel to recall that the full mass of the Monge-Ampère operator respects the partial order $\preccurlyeq$, i.e. that $V_{u}\leq V_{v}$ if $u\preccurlyeq v$ (\cite[Theorem 1.2]{WN17}). Moreover, assuming $\mathcal{A}\subset \mathcal{M}^{+}$ to be a totally ordered set of model type envelopes, its closure $\overline{\mathcal{A}}$ as subset of $PSH(X,\omega)$ (i.e. the weak closure) belongs to $\mathcal{M}$ and the Monge-Ampère operator becomes a homeomorphism when restricted to $\overline{\mathcal{A}}$ and when one considers the usual weak topologies (\cite[Lemma 3.12]{Tru20})
\subsection{The strong topology of finite energy $\omega$-psh functions.}
\label{ssec:Preli21}
A function $u\in PSH(X,\omega,\psi):=\{v\in PSH(X,\omega)\, :\, v\preccurlyeq \psi\}$ is said to have $\psi$-\emph{relative minimal singularities} if $u-\psi$ is globally bounded on $X$.
\begin{defn}[\cite{DDNL17b}]
The $\psi$-\emph{relative energy functional} $E_{\psi}:PSH(X,\omega,\psi)\to \mathbbm{R}\cup \{-\infty\}$ is defined as
$$
E_{\psi}(u):=\frac{1}{n+1}\sum_{j=0}^{n}\int_{X}(u-\psi) (\omega+dd^{c}u)^{j}\wedge(\omega+dd^{c}\psi)^{n-j}
$$
if $u$ has $\psi$-relative minimal singularities, and as
$$
E_{\psi}(u):=\inf\big\{E_{\psi}(v)\, :\, v\in\mathcal{E}(X,\omega,\psi)\, \mbox{with} \, \psi\mbox{-relative minimal singularities}, v\geq u\big\}
$$
otherwise. The subset $\mathcal{E}^{1}(X,\omega,\psi)\subset \mathcal{E}(X,\omega,\psi)$ is defined as
$$
\mathcal{E}^{1}(X,\omega,\psi):=\big\{u\in\mathcal{E}(X,\omega,\psi)\,:\, E_{\psi}(u)>-\infty\big\}.
$$
\end{defn}
Note that the $0-$relative energy functional is the \emph{Aubin-Mabuchi energy functional}, also called \emph{Monge-Ampère energy} (see \cite{Aub84}). As shown in \cite[Section 4.2]{DDNL17b}, $E_{\psi}$ is non-decreasing, continuous along decreasing sequences and $E_{\psi}(u)=\lim_{k\to \infty}E_{\psi}\big(\max(u,\psi-k)\big)$. The authors in \cite{DDNL17b} assumed $\psi$ to have \emph{small unbounded locus}, but all the mentioned properties extend to the general setting as an immediate consequence of the integration by parts formula proved in \cite[Theorem 1.1]{X19a} (see also \cite[Theorem 1.2]{Lu20}).\newline

The set $\mathcal{E}^{1}(X,\omega,\psi)$ for $\psi\in\mathcal{M}^{+}$ becomes a complete metric space when endowed with the distance
$$
d(u,v):=E_{\psi}(u)+E_{\psi}(v)-2E_{\psi}\big(P_{\omega}(u,v)\big)
$$
by \cite[Theorem A]{Tru19}. Moreover, assuming $\mathcal{A}\subset \mathcal{M}^{+}$ totally ordered, the metric spaces $\big(\mathcal{E}^{1}(X,\omega,\psi),d\big)$ can be glued together.
\begin{thm}[\cite{Tru19}, Theorem B]
The set $X_{\mathcal{A}}:=\bigsqcup_{\psi\in\overline{\mathcal{A}}}\mathcal{E}^{1}(X,\omega,\psi)$ can be endowed with a complete distance $d_{\mathcal{A}}$ such that $
d_{\mathcal{A}|\mathcal{E}^{1}(X,\omega,\psi)\times \mathcal{E}^{1}(X,\omega,\psi)}=d$ for any $\psi\in\overline{\mathcal{A}}$.
\end{thm}
For the purpose of this paper it is not important to recall the definition of the distance $d_{\mathcal{A}}$, which is quite technical, but the following interesting contraction properties will be crucial.
\begin{prop}[\cite{Tru19}, Lemma 4.2, Proposition 4.3]
\label{prop:PropProie}
Let $\psi_{1},\psi_{2},\psi_{3}\in\mathcal{M}$ such that $\psi_{1}\preccurlyeq\psi_{2}\preccurlyeq \psi_{3}$. Then $P_{\omega}[\psi_{1}]\big(P_{\omega}[\psi_{2}](u)\big)=P_{\omega}[\psi_{1}](u)$ for any $u\in\mathcal{E}^{1}(X,\omega,\psi_{3})$ and $|P_{\omega}[\psi_{1}](u)-\psi_{1}|\leq C$ if $|u-\psi_{3}|\leq C$. Moreover the map
$$
P_{\omega}[\psi_{1}](\cdot): \mathcal{E}^{1}(X,\omega,\psi_{2})\to PSH(X,\omega,\psi_{1})
$$
has image in $\mathcal{E}^{1}(X,\omega,\psi_{1})$ and it is a Lipschitz map of constant $1$ when the sets $\mathcal{E}^{1}(X,\omega,\psi_{i})$, $i=1,2$, are endowed with the $d$ distances, i.e.
$$
d\big(P_{\omega}[\psi_{1}](u), P_{\omega}[\psi_{1}](v)\big)\leq d(u,v)
$$
for any $u,v\in\mathcal{E}^{1}(X,\omega,\psi_{2})$.
\end{prop}
The metric topology of $\big(X_{\mathcal{A}}, d_{\mathcal{A}}\big)$ is truly natural as the next result recalls (see also \cite{BBEGZ16}).
\begin{prop}[\cite{Tru20}, Theorems $6.2$, $6.3$]
\label{prop:PropAB}
The metric topology of $\big(X_{\mathcal{A}}, d_{\mathcal{A}}\big)$ is the coarsest refinement of the weak topology such that $E_{\cdot}(\cdot)$ becomes continuous, i.e. given $\{u_{k}\}_{k\in\mathbbm{N}}, u\subset X_{\mathcal{A}}$ then the followings are equivalent:
\begin{itemize}
\item[i)] $d_{\mathcal{A}}(u_{k},u)\to 0$;
\item[ii)] $u_{k}\to u$ weakly and $E_{P_{\omega}[u_{k}]}(u_{k})\to E_{P_{\omega}[u]}(u)$.
\end{itemize}
Moreover if $d_{\mathcal{A}}(u_{k},u)\to 0$, then there exists a subsequence $\{u_{k_{j}}\}_{j\in\mathbbm{N}}$ such that $v_{j}:=(\sup_{h\geq j}u_{k_{h}})^{*}, w_{j}:=P_{\omega}(u_{k_{j}}, u_{k_{j+1}}, \dots)$ converge monotonically almost everywhere to $u$. In particular, the metric topology of $\big(X_{\mathcal{A}}, d_{\mathcal{A}}\big)$ implies the convergence in capacity.
\end{prop}
With obvious notations $P_{\omega}(u_{k_{j}}, u_{k_{j+1}},\dots):=\sup\{w\in PSH(X,\omega)\, : \, w\leq u_{k_{h}}\, \mbox{for any}\, h\geq j\}$, while a sequence $\{u_{k}\}_{k\in\mathbbm{N}}\subset PSH(X,\omega)$ is said to converge \emph{in capacity} to $u\in PSH(X,\omega)$ if for any $\delta>0$
$$
\mbox{Cap}\big(\{|u_{k}-u|\geq \delta\}\big)\to 0 
$$
as $k\to +\infty$, where for any $B\subset X$ Borel set
\begin{equation}
\label{eqn:Capacity}
\mbox{Cap}(B):=\sup\Big\{\int_{B}MA_{\omega}(u)\, : \, u\in PSH(X,\omega), -1\leq u\leq 0\Big\}
\end{equation}
(see \cite{Kol98}, \cite{GZ17} and reference therein).

We can now give the following key definition.
\begin{defn}[\cite{Tru20}\footnote{this definition is slightly different from the one given in \cite{Tru20} since the latter was introduced as the metric convergence of $\big(X_{\mathcal{A}},d_{\mathcal{A}}\big)$. However for the situations covered by this article the two definitions are equivalent by Proposition \ref{prop:PropAB}.}]
Let $u,\{u_{k}\}_{k\in\mathbbm{N}}\subset PSH(X,\omega)$ such that $u\in\mathcal{E}^{1}(X,\omega,P_{\omega}[u]), u_{k}\in\mathcal{E}^{1}(X,\omega,P_{\omega}[u_{k}])$ for any $k\in\mathbbm{N}$. We say that $u_{k}\to u$ \emph{strongly} if $u_{k}\to u$ weakly (i.e. in the $L^{1}$-topology) and $E_{P_{\omega}[u_{k}]}(u_{k})\to E_{P_{\omega}[u]}(u)$ as $k\to \infty$.
\end{defn}

The following Lemma gives a sufficient way to detect the strong convergence in many situations.
\begin{lem}
\label{lem:KeyConv}
Let $\psi_{k},\psi\in\mathcal{M}$ such that $\psi_{k}\to \psi$ monotonically almost everywhere. Let also $u_{k},v_{k}\in\mathcal{E}^{1}(X,\omega,\psi_{k})$ converging in capacity respectively to $u,v\in\mathcal{E}^{1}(X,\omega,\psi)$. Then for any $j=0,\dots,n$
\begin{equation}
    \label{eqn:FB1}
    (\omega+dd^{c}u_{k})^{j}\wedge (\omega+dd^{c}v_{k})^{n-j}\to (\omega+dd^{c}u)^{j}\wedge(\omega+dd^{c}v)^{n-j}
\end{equation}
weakly. Moreover, letting $\{f_{k}\}_{k\in\mathbbm{N}}, f$ be uniformly bounded quasi-continuous functions such that $f_{k}\to f$ in capacity and assuming that there exists a uniform constant $C$ such that $ u_{k},v_{k}\geq \psi_{k}-C$, the weak convergence
\begin{equation}
    \label{eqn:FB2}
    f_{k}(\omega+dd^{c}u_{k})^{j}\wedge (\omega+dd^{c}v_{k})^{n-j}\to f(\omega+dd^{c}u)^{j}\wedge(\omega+dd^{c}v)^{n-j}
\end{equation}
holds for any $j=0,\dots,n$. In particular, if $\{\psi_{k}\}_{k\in\mathbbm{N}}\subset \mathcal{M}^{+}$ is a totally ordered sequence such that $\psi_{k}\to \psi\in\mathcal{M}^{+}$ almost everywhere, then for any $u\in PSH(X,\omega)$ globally bounded,
\begin{equation}
    \label{eqn:FB3}
    P_{\omega}[\psi_{k}](u)\to P_{\omega}[\psi](u)
\end{equation}
strongly.
\end{lem}
\begin{proof}
The convergences (\ref{eqn:FB1}), (\ref{eqn:FB2}) are given by \cite[Proposition 2.7]{Tru19}.\newline
So, suppose to have a totally ordered sequence $\{\psi_{k}\}_{k\in\mathbbm{N}}\subset \mathcal{M}^{+}$ such that $\psi_{k}\to \psi\in\mathcal{M}^{+}$ almost everywhere, let $u\in PSH(X,\omega)$ globally bounded, and set $v_{k}:=P_{\omega}[\psi_{k}](u), v:=P_{\omega}[\psi](u)$. We claim that to prove the strong convergence $v_{k}\to v$, we can assume that $\{\psi_{k}\}_{k\in\mathbbm{N}}$ is monotone.\newline
Indeed, suppose that the strong convergence $P_{\omega}[\psi'_{k}](u)\to P_{\omega}[\psi'](u)$ holds for any monotone sequence $\{\psi'_{k}\}_{k\in\mathbbm{N}}\subset \mathcal{M}^{+}$ converging to $\psi'\in\mathcal{M}^{+}$, while by contradiction assume also that $v_{k}\nrightarrow v$ strongly (hence in this case $\psi_{k}$ is necessarily not monotone). Then there exists $\epsilon>0$ and a subsequence $\{v_{k_{j}}\}_{j\in\mathbbm{N}}$ such that $d_{\mathcal{A}}(v_{k_{j}},v)\geq \epsilon$ for any $j\in\mathbbm{N}$ by Proposition \ref{prop:PropAB}. In particular, since $\{\psi_{k}\}_{k\in\mathbbm{N}}$ is totally ordered, there exists a further monotone subsequence $\{v_{k_{j_{h}}}\}_{h\in\mathbbm{N}}$ such that $d_{\mathcal{A}}(v_{k_{j_{h}}},v)\geq \epsilon$ for any $h\in\mathbbm{N}$, which is clearly in contradiction with $v_{k_{j_{h}}}\to v$ strongly. Hence the claim is proved.\newline
Therefore, letting $\{\psi_{k}\}_{k\in\mathbbm{N}}\subset \mathcal{M}^{+}$ be a  monotone sequence converging to $\psi\in\mathcal{M}^{+}$ almost everywhere, it remains to prove that $v_{k}:=P_{\omega}[\psi_{k}](u)\to v:=P_{\omega}[\psi](u)$ strongly for $u\in PSH(X,\omega)$ globally bounded. Moreover, since $|v_{k}-\psi_{k}|\leq C$ uniformly by Proposition \ref{prop:PropProie}, the definition of the Monge-Ampère energy and the convergence (\ref{eqn:FB2}) imply that $v_{k}\to v$ strongly if and only if $v_{k}\to v$ weakly.\newline
Assume $\psi_{k}\searrow \psi$, and let $\tilde{v}:=\lim_{k\to +\infty}v_{k}$. Clearly $v_{k}\searrow \tilde{v}\geq v$, so it remains to prove that $\tilde{v}\leq v$. But combining $u\geq \tilde{v}\geq v$ with $MA_{\omega}(v)\leq \mathbbm{1}_{\{u=v\}}MA_{\omega}(u)$ (\cite[Theorem 3.8]{DDNL17b}), we obtain
$$
0\leq \int_{X}(\tilde{v}-v)MA_{\omega}(v)\leq \int_{\{v=u\}}(\tilde{v}-u)MA_{\omega}(u)\leq 0,
$$
i.e. $\tilde{v}=v$ $MA_{\omega}(v)$-almost everywhere. Note also that $|\tilde{v}-\psi|\leq C$ since $|v_{k}-\psi_{k}|\leq C$. Thus, by the domination principle in the class $\mathcal{E}(X,\omega,\psi)$ (see \cite[Proposition 3.11]{DDNL17b}) we get $\tilde{v}\leq v$ which concludes this case.\newline
Assume $\psi_{k}\nearrow \psi$. Then, with the same notations of before, $v_{k}\nearrow \tilde{v}\leq v $ and again $|\tilde{v}-\psi|\leq C$. Thus (\ref{eqn:FB1}) gives $MA_{\omega}(v_{k})\to MA_{\omega}(\tilde{v})$ weakly (monotone convergence implies the convergence in capacity, see \cite[Proposition 4.25]{GZ17}), which together with $MA_{\omega}(v_{k})\leq \mathbbm{1}_{\{v_{k}=u\}}MA_{\omega}(u)$ (\cite[Theorem 3.8]{DDNL17b}) implies that
$$
MA_{\omega}(\tilde{v})\leq \mathbbm{1}_{\{v=u\}}MA_{\omega}(u).
$$
Therefore,
$$
0\leq \int_{X}(v-\tilde{v})MA_{\omega}(\tilde{v})\leq \int_{\{\tilde{v}=u\}}(v-u)MA_{\omega}(u)\leq 0,
$$
which by the domination principle in the class $\mathcal{E}(X,\omega,\psi)$ (\cite[Proposition 3.11]{DDNL17b}) yields $v\leq \tilde{v}$ and concludes the proof.
\end{proof}
Clearly, Lemma \ref{lem:KeyConv} implies that for any Kähler potential $\varphi$, i.e. for any element of $ \mathcal{H}:=\{\varphi\in PSH(X,\omega)\cap C^{\infty}(X)\, : \, \omega+dd^{c}\varphi>0\}$, and for any totally ordered sequence $\{\psi_{k}\}_{k\in\mathbbm{N}}\subset \mathcal{M}^{+}$ converging to $\psi\in\mathcal{M}^{+}$, $P_{\omega}[\psi_{k}](\varphi)\to P_{\omega}[\psi](\varphi)$ strongly. This example of strong convergence will be heavily used in the subsequent proofs, as we also recall that any element $u\in PSH(X,\omega)$ can be approximated by a decreasing sequence of Kähler potentials (\cite[Theorem 1]{BK07}).\newline

Finally the following essential property of the energy $E_{\cdot}(\cdot)$ in $X_{\mathcal{A}}$ will be a powerful tool for the variational approach.
\begin{prop}[\cite{Tru20}, Lemma $3.13$, Propositions $3.14$, $3.15$]
\label{prop:Usc}
Let $\mathcal{A}\subset \mathcal{M}^{+}$ be a totally ordered family such that $\overline{A}\subset \mathcal{M}^{+}$, and let $\{u_{k}\}_{k\in\mathbbm{N}}\subset X_{\mathcal{A}}$ converging weakly to $u\in X_{\mathcal{A}}$. Then
$$
\limsup_{k\to \infty}E_{P_{\omega}[u_{k}]}(u_{k})\leq E_{P_{\omega}[u]}(u).
$$
Moreover if $E_{P_{\omega}[u_{k}]}(u_{k})\geq -C$ uniformly, then $P_{\omega}[u_{k}]\to P_{\omega}[u]$ weakly. In particular for any $C\in\mathbbm{N}$ the set
$$
X_{\mathcal{A},C}:=\{u\in X_{\mathcal{A}}\, : \, \sup_{X}u\leq C \, \mbox{and}\, E_{P_{\omega}[u]}(u)\geq -C\}
$$
is weakly compact.
\end{prop}
\subsection{The strong topology of finite energy non-pluripolar measures.}
\label{ssec:EnergyMeasure}
On the set of probability measures, the Monge-Ampère counterpart of the $\psi$-relative energy $E_{\psi}(\cdot)$ and of the associated set $\mathcal{E}^{1}(X,\omega,\psi)$ for $\psi\in\mathcal{M}^{+}$ are respectively the $\psi$-relative energy $E_{\psi}^{*}$ and the set $\mathcal{M}^{1}(X,\omega,\psi)$. \\
For $\mu$ positive probability measure, the first one is defined as
$$
E^{*}_{\psi}(\mu):=\sup_{\mathcal{E}^{1}(X,\omega,\psi)}F_{\mu,\psi}:=\sup_{u\in\mathcal{E}^{1}(X,\omega,\psi)} \Big(E_{\psi}(u)-V_{\psi}L_{\mu}(u) \Big)\in [0,\infty]
$$
where $V_{\psi}:=\int_{X}MA_{\omega}(\psi)>0$ and where $L_{\mu}(u):=\lim_{k\to \infty}\int_{X}\big(\max(u,\psi-k)-\psi\big)d\mu$ if $\mu$ does not charge $\{\psi=-\infty\}$ and $L_{\mu}\equiv-\infty$ otherwise (see \cite[Section 4]{Tru20}). The maximizers of the translation invariant functional $F_{\mu,\psi}$ solve the Monge-Ampère equation $MA_{\omega}(u)=V_{\psi}\mu$ (\cite[Proposition 5.2]{Tru20}) and, defining
$$
\mathcal{M}^{1}(X,\omega,\psi):=\{V_{\psi}\mu\, :\, \mu \, \mbox{probabilty measure such that}\, E_{\psi}^{*}(\mu)<\infty\},
$$
and $\mathcal{E}^{1}_{norm}(X,\omega,\psi):=\{u\in\mathcal{E}^{1}(X,\omega,\psi)\, : \, \sup_{X}u=0\}$, the following result holds.
\begin{thm}[\cite{Tru20}, Theorem A]
\label{thm:A}
The Monge-Ampère map $MA_{\omega}:\big(\mathcal{E}^{1}_{norm}(X,\omega,\psi),d\big)\to \big(\mathcal{M}^{1}(X,\omega,\psi),strong\big)$ is a homeomorphism where the \emph{strong topology} on $\mathcal{M}^{1}(X,\omega,\psi)$ is the coarsest refinement of the weak topology such that $E_{\psi}^{*}$ becomes continuous. In particular, $u\in\mathcal{E}^{1}(X,\omega,\psi)$ solves the Monge-Ampère equation $MA_{\omega}(u)=V_{\psi}\mu$ if and only if it maximizes the functional $F_{\mu,\psi}$.
\end{thm}
Examples of non-pluripolar measures $\nu$ with $\nu(X)=V_{\psi}$ such that $\nu\in \mathcal{M}^{1}(X,\omega,\psi)$ are given by $\nu=f\omega^{n}$ for $f\in L^{p}$, $p>1$ (\cite[Theorem 1.4]{DDNL17b}).\newline

More generally, given $\mathcal{A}\subset \mathcal{M}^{+}$ totally ordered such that $\overline{\mathcal{A}}\subset \mathcal{M}^{+}$ and endowing the set
$$
Y_{\mathcal{A}}:=\bigsqcup_{\psi\in \overline{\mathcal{A}}}\mathcal{M}^{1}(X,\omega,\psi)
$$
with the \emph{strong topology} given as the coarsest refinement of the weak topology of measures such that $E_{\cdot}^{*}(\cdot)$ becomes continuous, we get the following Theorem.
\begin{thm}[\cite{Tru20}, Theorem B]
\label{thm:ThmA}
The Monge-Ampère map
$$
MA_{\omega}:\big(X_{\mathcal{A},norm},d_{\mathcal{A}}\big)\to (Y_{\mathcal{A}}, strong)
$$
is a homeomorphism where $X_{\mathcal{A},norm}:=\bigsqcup_{\psi\in\overline{\mathcal{A}}}\mathcal{E}^{1}_{norm}(X,\omega,\psi)$.
\end{thm}
\section{Strong continuity of solutions.}
\label{sec:Degenerate}
As stated in the Introduction given a totally ordered sequence $\psi_{k}\in\mathcal{M}^{+}$ converging weakly to $\psi\in\mathcal{M}^{+}$, and given $f_{k}\in L^{1}\setminus \{0\}$ non-negative functions $L^{1}$-converging to $f\in L^{1}\setminus\{0\}$ we want to give necessary conditions so that a sequence of solutions $\{u_{k}\}_{k\in\mathbbm{N}}$ of
\begin{equation}
\label{eqn:MAKEk}
\begin{cases}
MA_{\omega}(u_{k})=e^{-\lambda u_{k}}f_{k}\omega^{n}\\
u_{k}\in\mathcal{E}^{1}(X,\omega,\psi_{k})
\end{cases}
\end{equation}
converges strongly in $X_{\mathcal{A}}$ for $\mathcal{A}:=\{\psi_{k}\}_{k\in\mathbbm{N}}$ to a solution $u$ of
\begin{equation}
\label{eqn:MAKE_}
\begin{cases}
MA_{\omega}(u)=e^{-\lambda u}f\omega^{n}\\
u\in\mathcal{E}^{1}(X,\omega,\psi).
\end{cases}
\end{equation}
We have three very different cases based on the sign of $\lambda\in\mathbbm{R}$.
\subsection{Case $\lambda=0$.}
\label{ssec:Zero}
In this subsection $\lambda=0$.\\
Observe that by Theorem \ref{thm:A} the existence of the solutions $u_{k}$ of (\ref{eqn:MAKEk}), normalized by $\sup_{X}u_{k}=0$, is equivalent to $f_{k}\omega^{n}\in\mathcal{M}^{1}(X,\omega,\psi_{k})$.\newline

We recall that the set $PSH_{norm}(X,\omega):=\{u\in PSH(X,\omega)\, : \, \sup_{X}u=0\}$ is wekly compact (see \cite[Proposition 8.5]{GZ17}).
\begin{reptheorem}{thmA}
Let $f_{k},f\in L^{1}$, $\psi_{k},\psi\in \mathcal{M}^{+}$, $u_{k}\in\mathcal{E}^{1}(X,\omega,\psi_{k})$ as in the aforementioned hypothesis. Let also $u\in PSH(X,\omega)$ be the unique accumulation point for $\{u_{k}\}_{k\in\mathbbm{N}}$. Then $u\in\mathcal{E}_{norm}(X,\omega,\psi)$ and $MA_{\omega}(u)=f\omega^{n}$. Furthermore, $u\in\mathcal{E}^{1}_{norm}(X,\omega,\psi)$ and $u_{k}\to u$ strongly if and only if $E_{\psi_{k}}(u_{k})\geq -C$ for a uniform constant $C\geq 0$ and
\begin{equation}
\label{eqn:WeirdH_}
\limsup_{k\to \infty}\int_{X}(\psi_{k}-u_{k})f_{k}\omega^{n}\leq \int_{X}(\psi-u)f\omega^{n}.
\end{equation}
\end{reptheorem}
\begin{proof}
\cite[Lemma 2.8]{DDNL18b} immediately implies $MA_{\omega}(u)\geq f\omega^{n}$.\newline
As recalled in the beginning of section \S $2$, the Monge-Ampère mass respects the partial order $\preccurlyeq$ by the main result in \cite{WN17} and the Monge-Ampère operator becomes an homeomorphism (with respect to the weak topologies) when restricted to a totally ordered set $\mathcal{A}\subset \mathcal{M}^{+}$ (\cite[Lemma 3.12]{Tru20}). Thus, as $\sup_{X}(u-\psi)=\sup_{X}u=0$ by Hartogs' Lemma (\cite[Proposition 8.4]{GZ17}), we get $V_{u}\leq V_{\psi}$, while the weak convergences $\psi_{k}\to \psi, f_{k}\to f$ implies
$$
V_{\psi}=\lim_{k\to \infty}V_{\psi_{k}}=\lim_{k\to\infty}\int_{X}f_{k}\omega^{n}=\int_{X}f\omega^{n}.
$$
Hence $u$ solves $MA_{\omega}(u)=f\omega^{n}$, i.e. it is the unique solution in $\mathcal{E}_{norm}(X,\omega,\psi)$, which concludes the first part of the proof.\newline
Assume that $u\in\mathcal{E}^{1}_{norm}(X,\omega,\psi)$ and that $u_{k}\to u$ strongly. Then $E_{\psi_{k}}(u_{k})\to E_{\psi}(u)$, which clearly gives $E_{\psi_{k}}(u_{k})\geq -C$ uniformly. Moreover, by Theorem \ref{thm:ThmA}
$$
E_{\psi_{k}}^{*}\big(MA_{\omega}(u_{k})/V_{\psi_{k}}\big)=E_{\psi_{k}}(u_{k})-\int_{X}(u_{k}-\psi_{k})f_{k}\omega^{n}\longrightarrow E_{\psi}^{*}\big(MA_{\omega}(u)/V_{\psi}\big)=E_{\psi}(u)-\int_{X}(u-\psi)f\omega^{n}
$$
which clearly gives $\int_{X}(\psi_{k}-u_{k})f_{k}\omega^{n}\to \int_{X}(\psi-u)f\omega^{n}$ by definition of $L_{\mu}$.\newline
Conversely, suppose that $E_{\psi_{k}}(u_{k})\geq-C$ uniformly and that $\limsup_{k\to \infty}\int_{X}(\psi_{k}-u_{k})f_{k}\omega^{n}\leq \int_{X}(\psi-u)f\omega^{n}$. Then, for any $\varphi\in\mathcal{H}$ (i.e. a Kähler potential),
\begin{multline}
\label{eqn:P1}
\liminf_{k\to \infty} E_{\psi_{k}}^{*}\big(MA_{\omega}(u_{k})/V_{\psi_{k}}\big)\geq \liminf_{k\to \infty}\Big(E_{\psi_{k}}\big(P_{\omega}[\psi_{k}](\varphi)\big)+\int_{X}\big(\psi_{k}-P_{\omega}[\psi_{k}](\varphi)\big)f_{k}\omega^{n}\Big)\geq\\
\geq E_{\psi}\big(P_{\omega}[\psi](\varphi)\big)+\int_{X}\big(\psi-P_{\omega}[\psi](\varphi)\big)f\omega^{n}
\end{multline}
where the first inequality follows because $u_{k}$ is a maximizer of $F_{f_{k}\omega^{n}/V_{\psi_{k}},\psi_{k}}$ by Theorem \ref{thm:A}, the convergence $E_{\psi_{k}}\big(P_{\omega}[\psi_{k}](\varphi)\big)\to E_{\psi}\big(P_{\omega}[\psi](\varphi)\big)$ is given by Lemma \ref{lem:KeyConv} while the lower semicontinuity of the integral is a consequence of the Fatou's Lemma as $\big(\psi_{k}-P_{\omega}[\psi_{k}](\varphi)\big)f_{k}\to \big(\psi-P_{\omega}[\psi](\varphi)\big)f$ almost everywhere and $|\psi_{k}-P_{\omega}[\psi_{k}](\varphi)|\leq C$ uniformly (Proposition \ref{prop:PropProie}). Therefore, for any $v\in\mathcal{E}^{1}(X,\omega,\psi)$ letting $\varphi_{j}\in\mathcal{H}$ be a decreasing sequence converging to $v$ (\cite[Theorem 1]{BK07}), we get
\begin{multline}
\label{eqn:P2}
\liminf_{k\to \infty} E_{\psi_{k}}^{*}\big(MA_{\omega}(u_{k})/V_{\psi_{k}}\big)\geq \limsup_{j\to\infty}\Big(E_{\psi}\big(P_{\omega}[\psi](\varphi_{j})\big)+\int_{X}\big(\psi-P_{\omega}[\psi](\varphi_{j})\big)f\omega^{n}\Big)=E_{\psi}(v)+\int_{X}(\psi-v)f\omega^{n}
\end{multline}
exploiting the continuity of $E_{\psi}(\cdot)$ along decreasing sequences and the Monotone Convergence Theorem. Thus, by definition,
\begin{equation}
\label{eqn:P4}
\liminf_{k\to \infty}E_{\psi_{k}}^{*}\big(MA_{\omega}(u_{k})/V_{\psi_{k}}\big)\geq E^{*}_{\psi}\big(f\omega^{n}/V_{\psi}\big)=E_{\psi}^{*}\big(MA_{\omega}(u)/V_{\psi}\big),
\end{equation}
which together with $\int_{X}(\psi_{k}-u_{k})f_{k}\omega^{n}\to \int_{X}(\psi-u)f\omega^{n}$ (by Fatou's Lemma and (\ref{eqn:WeirdH_})) and the upper semicontinuity of $E_{\cdot}(\cdot)$ (Proposition \ref{prop:Usc}) imply $E_{\psi_{k}}(u_{k})\to E_{\psi}(u)$. Hence $u\in\mathcal{E}^{1}_ {norm}(X,\omega,\psi)$ and $u_{k}\to u$ strongly.
\end{proof}
\begin{rem}
\emph{It is clear from the proof of Theorem \ref{thmA} that to prove that $u\in\mathcal{E}^{1}_{norm}(X,\omega,\psi)$ it is enough to show that $E_{\psi_{k}}(u_{k})\geq-C$. Therefore Theorem \ref{thmA} gives an effective tool to find out if the unique solution $u\in\mathcal{E}_{norm}(X,\omega,\psi)$ belongs to $\mathcal{E}^{1}(X,\omega,\psi)$, which is a regularity result.}
\end{rem}
\begin{rem}
\label{rem:SomeCases}
\emph{There are many examples from which the assumptions on the boundedness of the energy and $(\ref{eqn:WeirdH_})$ in Theorem \ref{thmA} are satisfied. \\
For instance if there exists $h\in L^{1}$ such that $(\psi_{k}-u_{k})f_{k}\leq h$ almost everywhere for any $k\in\mathbbm{N}$ then (\ref{eqn:WeirdH_}) trivially holds, while by \cite[Theorem 4.10]{DDNL17b} $-E_{\psi_{k}}(u_{k})\leq \int_{X}(\psi_{k}-u_{k})f_{k}\omega^{n}\leq ||h||_{L^{1}}$.\\
Similarly if $||f_{k}||_{L^{p}},||f||_{L^{p}}$ are uniformly bounded for $p>1$, then the boundedness of the energy and (\ref{eqn:WeirdH_}) are consequences of $\psi_{k}-u_{k}\to \psi-u$ in $L^{r}$ for any $r\in [1,\infty)$ (see Theorem $1.48$ in \cite{GZ17}). In particular, Theorem \ref{thmA} extends \cite[Theorem C]{Tru20} and \cite[Theorem 1.4]{DDNL19} \\
Next, if $f_{k}\leq cf$ for a constant $c\in\mathbbm{R}$ and $E_{\psi_{k}}(u_{k})\geq -C$ uniformly, then (\ref{eqn:WeirdH_}) can be replaced by $\int_{B}f\omega^{n}\leq A\, \mbox{Cap}(B)$ for any Borel set $B\subset X$ for $A>0$. Indeed, as all the \emph{$\psi$-relative Monge-Ampère capacities} are comparable (\cite[Theorem 1.1]{Lu20}), it is not difficult to check that there exists an uniform constant $A'>0$ such that $\int_{B}f\omega^{n}\leq A'\mbox{Cap}_{\psi_{k}}(B)$ for any Borel set $B\subset X$ and any $k\in\mathbbm{N}$ (see the proof of \cite[Lemma 2.8]{Tru19}). Thus by \cite[Lemma 4.18]{DDNL17b} we get that $\int_{X}(\psi_{k}-u_{k})^{2}f\omega^{n}\leq C_{1}\big(|E_{\psi_{k}}(u_{k})|+1\big)\leq C_{2}$ for two uniform constants $C_{1}, C_{2}$. In particular $\{\psi_{k}-u_{k}\}_{k\in\mathbbm{N}}\cup \{\psi-u\}\subset L^{1}(f\omega^{n})$ is uniformly integrable, and by \cite[Theorem 4.4]{Tru20} it follows that $\int_{X}(\psi_{k}-u_{k})f\omega^{n}\to \int_{X}(\psi-u)f\omega^{n}$. Moreover, as $f_{k}\leq cf$, by an easy calculation there exist uniform constants $C_{1},C_{2}$ such that for any $\epsilon>0$
\begin{multline*}
    \int_{X}(\psi_{k}-u_{k})(f_{k}-f)\omega^{n}\leq C_{1}\epsilon\int_{\{f\leq\epsilon\}}(\psi_{k}-u_{k})\omega^{n}+\int_{\{f>\epsilon\}}(\psi_{k}-u_{k})\Big(\frac{f_{k}}{f}-1\Big)f\omega^{n}\leq\\
    \leq C_{1}\epsilon+C_{2}\Big(\int_{\{f>\epsilon\}}\frac{(f_{k}-f)^{2}}{f}\omega^{n}\Big)^{1/2}\longrightarrow C_{1}\epsilon,
\end{multline*}
where the convergence as $k\to \infty$ follows by Lebesgue's Dominated Convergence Theorem. Summarizing, we get
$$
\limsup_{k\to \infty}\int_{X}(\psi_{k}-u_{k})f_{k}\omega^{n}\leq\limsup_{k\to \infty}\int_{X}(\psi_{k}-u_{k})(f_{k}-f)\omega^{n}+\limsup_{k\to\infty}\int_{X}(\psi_{k}-u_{k})f\omega^{n}=\int_{X}(\psi-u)f\omega^{n}.
$$
Similarly, if there exists a bounded sequence $\{c_{k}\}_{k\in\mathbbm{R}_{>0}}$ such that $c_{k}f_{k}$ becomes decreasing and $E_{\psi_{k}}(u_{k})\geq -C$ uniformly, then (\ref{eqn:WeirdH_}) can be replaced by $\int_{B}f_{k_{0}}\omega^{n}\leq A\, \mbox{Cap}(B)$ for any Borel set $B\subset X$ for $A>0$ and $k_{0}\in\mathbbm{N}$.
}
\end{rem}
\subsection{Case $\lambda<0$.}
Letting $f\in L^{1}\setminus\{0\}$ non negative and $\lambda\in\mathbbm{R}\setminus \{0\}$, we introduce the functional $L_{f,\lambda}:PSH(X,\omega)\to \overline{\mathbbm{R}}$ as
$$
L_{f,\lambda}(u):=\frac{-1}{\lambda}\log\int_{X}e^{-\lambda u}f\omega^{n}.
$$
Thus, for $\psi\in\mathcal{M}^{+}$, we define the functional $F_{f,\psi,\lambda}:\mathcal{E}^{1}(X,\omega,\psi)\to \overline{\mathbbm{R}}$ as $F_{f,\psi,\lambda}(u):=\big(E_{\psi}-V_{\psi} L_{f,\lambda}\big)(u)$ that must not be confused with the functional $F_{\mu,\psi}$ defined in section \ref{sec:Pre}. It is easy to see that $F_{f,\psi,\lambda}$ is invariant by translation, i.e. it descends to the space of currents. Moreover its maximizers solve a complex Monge-Ampère equation as the next result recalls.
\begin{thm}[\cite{DDNL17b}, Theorem $4.22$]
\label{thm:Max}
Let $f\in L^{1}\setminus\{0\}$ non negative and $\lambda\neq 0$. If $u\in \mathcal{E}^{1}(X,\omega,\psi)$ maximizes $F_{f,\psi,\lambda}$ then $u$ solves
\begin{equation}
\label{eqn:KELP}
\begin{cases}
MA_{\omega}(u)=e^{-\lambda u+C}f\omega^{n}\\
u\in\mathcal{E}^{1}(X,\omega,\psi)
\end{cases}
\end{equation}
for a constant $C\in\mathbbm{R}$.
\end{thm}
From now on until the end of this subsection we will assume $\lambda<0$. In this case the converse of Theorem \ref{thm:Max} holds and (\ref{eqn:KELP}) is solvable.
\begin{thm}[\cite{DDNL17b}, Theorem $4.23$, Lemma $4.24.$]
\label{thm:PositiveCase}
Let $\lambda<0$ and let $f\in L^{1}\setminus\{0\}$ non negative. Then the complex Monge-Ampère equation (\ref{eqn:KELP}) admits a unique solution and it maximizes $F_{f,\psi,\lambda}$ over $\mathcal{E}^{1}(X,\omega,\psi)$.
\end{thm}
A key result for the proof of Theorem \ref{thm:PositiveCase} is the following Lemma that will be also essential for our Theorem \ref{thmB}.
\begin{lem}
\label{lem:11.5GZ}
Let $g_{k},g$ non-negative $L^{1}$-functions such that $g_{k}\to g$ in $L^{1}$, and let $u,\{u_{k}\}_{k\in\mathbbm{N}}\subset PSH(X,\omega)$ such that $u_{k}\to u$ weakly. Then
$$
\int_{X}e^{u_{k}}g_{k}\omega^{n}\to \int_{X}e^{u}g\omega^{n}
$$
as $k\to \infty$
\end{lem}
\begin{proof}
By Fatou's Lemma $\liminf_{k\to +\infty}\int_{X}e^{u_{k}}g_{k}\omega^{n}\geq \int_{X}e^{u}g\omega^{n}$. Thus, as
$$
\int_{X}e^{u_{k}}g_{k}\omega^{n}\leq e^{\sup_{X}u_{k}}\int_{X}|g_{k}-g|\omega^{n}+\int_{X}e^{u_{k}}g\omega^{n},
$$
the result follows from $\sup_{X}u_{k}\leq C$ and Lebesgue's Dominated Convergence Theorem.
\end{proof}
We can now prove Theorem \ref{thmB}.
\begin{reptheorem}{thmB}
Assume
\begin{itemize}
\item[(i)] $\lambda<0$;
\item[(ii)] $f_{k},f\in L^{1}\setminus\{0\}$ non-negative functions such that $f_{k}\to f$ in $L^{1}$;
\item[(iii)] $\{\psi_{k}\}_{k\in \mathbbm{N}}\subset\mathcal{M}^{+}$ totally ordered such that $\psi_{k}\to\psi\in\mathcal{M}^{+}$ weakly.
\end{itemize}
Let $u_{k}\in\mathcal{E}^{1}(X,\omega,\psi_{k})$, $u\in\mathcal{E}^{1}(X,\omega,\psi)$ be the unique solutions respectively of
\begin{equation}
\begin{cases}
MA_{\omega}(u_{k})=e^{-\lambda u_{k}}f_{k}\omega^{n}\\
u_{k}\in\mathcal{E}^{1}(X,\omega,\psi_{k}),
\end{cases}
\begin{cases}
MA_{\omega}(u)=e^{-\lambda u}f\omega^{n}\\
u\in\mathcal{E}^{1}(X,\omega,\psi).
\end{cases}
\end{equation}
Then $u_{k}\to u$ strongly.
\end{reptheorem}
\begin{proof}
Assume $\lambda=-1$ for simplicity of notations.\newline 
Observe that by an easy contradiction argument it is enough to check that any subsequence $\{u_{k_{j}}\}_{j\in\mathbbm{N}}$ admits a further subsequence $\{u_{k_{j_{h}}}\}_{h\in\mathbbm{N}}$ converging strongly to $u$. Indeed, assuming that this happens, if by contradiction $u_{k}\nrightarrow u$ strongly then there exists a subsequence $\{u_{k_{j}}\}_{j\in\mathbbm{N}}$, and $\epsilon>0$ such that $d_{\mathcal{A}}(u_{k_{j}},u)\geq \epsilon$ for any $j\in\mathbbm{N}$ (by Proposition \ref{prop:PropAB}). Thus, the contradiction is given by extracting a subsequence $\{u_{k_{j_{h}}}\}_{h\in\mathbbm{N}}$ that strongly converges to $u$.\newline
So without loss of generality we may assume $u_{k_{j}}$ to be the whole sequence and we set $F_{k}:=F_{f_{k},\psi_{k},-1}$, $F:=F_{f,\psi,-1}$. By Theorem \ref{thm:PositiveCase} $u_{k}$ maximizes $F_{k}$ for any $k\in\mathbbm{N}$ while $u$ maximizes $F$. Thus, for any $\varphi\in\mathcal{H}$,
$$
\liminf_{k\to \infty} F_{k}(u_{k})\geq \liminf_{k\to \infty} F_{k}\big(P_{\omega}[\psi_{k}](\varphi)\big)=F\big(P_{\omega}[\psi](\varphi)\big)
$$
where the convergence follows from Lemmas \ref{lem:KeyConv} and \ref{lem:11.5GZ}. Hence, passing to the supremum over $\mathcal{H}$, combining \cite[Theorem 1]{BK07}, the continuity of $E_{\psi}$ along decreasing sequences and again Lemma \ref{lem:11.5GZ} we get
\begin{equation}
\label{eqn:LIMINF}
\liminf_{k\to \infty}F_{k}(u_{k})\geq \sup_{\varphi\in\mathcal{H}}F\big(P_{\omega}[\psi](\varphi)\big)=\sup_{v\in\mathcal{E}^{1}(X,\omega,\psi)}F(v)= F(u).
\end{equation}
Moreover, up to considering a subsequence, the sequence $v_{k}:=u_{k}-\sup_{X}u_{k}\leq \psi_{k}$ converges weakly to $v\in PSH(X,\omega), v\leq \psi$ and Lemma \ref{lem:11.5GZ} yields
$$
a_{k}:=\int_{X}e^{v_{k}}f_{k}\omega^{n}\to \int_{X}e^{v}f\omega^{n}\in \big(0,||f||_{L^{1}}\big].
$$
Hence, using the complex Monge-Ampère equations,
$$
\sup_{X}u_{k}=\log V_{\psi_{k}}-\log a_{k}
$$
is uniformly bounded (recall that $V_{\psi_{k}}\to V_{\psi}>0$ as a consequence of the monotonicity of $\{\psi_{k}\}_{k\in\mathbbm{N}}$, see section \S \ref{sec:Pre}), and $\{u_{k}\}_{k\in \mathbbm{N}}$, up to considering a subsequence, converges weakly to $\tilde{u}\in PSH(X,\omega), \tilde{u}\preccurlyeq \psi$. On the other hand from (\ref{eqn:LIMINF}) and the triangle inequality on $\big(\mathcal{E}^{1}(X,\omega,\psi_{k}),d\big)$, as $\sup_{X}u_{k}\leq A$ for $A\geq 0$, we have
\begin{multline*}
\limsup_{k\to \infty}d(\psi_{k},u_{k})\leq \limsup_{k\to\infty}\big(d(\psi_{k},u_{k}-A)+d(u_{k},u_{k}-A)\big)=2AV_{\psi}-\liminf_{k\to \infty}E_{\psi_{k}}(u_{k})\leq\\
\leq2AV_{\psi}-F(u)-\limsup_{k\to \infty}V_{\psi_{k}}\log\int_{X}e^{u_{k}}f_{k}\omega^{n}\leq -F(u)+C    
\end{multline*}
Therefore by Proposition \ref{prop:Usc} and again Lemma \ref{lem:11.5GZ} we obtain $\tilde{u}\in \mathcal{E}^{1}(X,\omega,\psi)$ and
$$
\limsup_{k\to \infty}F_{k}(u_{k})\leq F(\tilde{u}),
$$
which necessarily implies $\tilde{u}=u+B$ for a constant $B\in\mathbbm{R}$ (Theorem \ref{thm:PositiveCase}) as $F(u)\leq F(\tilde{u})$ by (\ref{eqn:LIMINF}). However from the Monge-Ampère equations it follows that
$$
e^{B}\int_{X}e^{u}f\omega^{n}=\int_{X}e^{\tilde{u}}f\omega^{n}=\lim_{k\to \infty}\int_{X}e^{u_{k}}f_{k}\omega^{n}=\lim_{k\to \infty}V_{\psi_{k}}=V_{\psi}=\int_{X}e^{u}f\omega^{n},
$$
i.e. $B=0$. In conclusion we have proved that $u_{k}\to u$ weakly, $F_{k}(u_{k})\to F(u)$ and $\int_{X}e^{u_{k}}f_{k}\omega^{n}\to \int_{X}e^{u}f\omega^{n}$. Hence $E_{\psi_{k}}(u_{k})\to E_{\psi}(u)$, which concludes the proof.
\end{proof}
\begin{rem}
\emph{Theorem \ref{thmB} extends to the general case of positive non-pluripolar measures $\mu$ of finite mass (instead of $\omega^{n}$), assuming $f_{k}\to f$ in $L^{1}(\mu)$, as the analogues of Theorem \ref{thm:PositiveCase} and Lemma \ref{lem:11.5GZ} hold.}
\end{rem}
\subsection{Case $\lambda>0$.}
If $\lambda>0$ then the study of (\ref{eqn:MAKE_}) is much more complex that the case $\lambda\leq 0$ even in the absolute setting $\psi=0$. As stated in the Introduction, for instance, if $\{\omega\}=-K_{X}$, i.e. $X$ is a \emph{Fano manifold}, $f\equiv 1$ and $\psi=0$, the existence of a solution is characterized by an algebro-geometric notion called \emph{K-stability} (see \cite{CDS15}). The uniqueness of solutions of (\ref{eqn:MAKE_}) is a hard problem as well (see \cite{BM86}). Note also that in this case $F_{1,0,1}$ (where we recall that $F_{f,\psi,\lambda}:=E_{\psi}-V_{\psi}L_{f,\lambda}$ for $f\in L^{1}$, $\lambda\in \mathbbm{R}\setminus \{0\}$, $\psi\in\mathcal{M}^{+}$) coincides with the Ding functional (\cite{Ding88}). We refer to the companion paper \cite{Tru20b} where we analyze this interesting situation more in detail.

To prove Theorems \ref{thmC} and \ref{thmD} we will use the functional $J_{\psi}:\mathcal{E}^{1}(X,\omega,\psi)\to \mathbbm{R}$,
$$
J_{\psi}(u):=-E_{\psi}(u)+\int_{X}(u-\psi)MA_{\omega}(\psi)
$$
for $\psi\in\mathcal{M}^{+}$ (see \cite{Tru20}, being aware that the notation is slightly different). It is immediate to check that $J_{\psi}$ is non-negative and translation invariant. Indeed it represents the translation invariant version of the distance $d$ as the following key lemma shows.
\begin{lem}
\label{lem:J}
Let $\psi\in\mathcal{M}^{+}$. Then there exists $C\in\mathbbm{R}_{\geq 0}$ depending only on $(X,\omega)$ such that
$$
d(u,\psi)-C\leq J_{\psi}(u)\leq d(u,\psi)
$$
for any $u\in\mathcal{E}^{1}_{norm}(X,\omega,\psi)$.
\end{lem}
\begin{proof}
Since $u\leq \psi$ we have $d(u,\psi)=-E_{\psi}(u)\geq -E_{\psi}(u)+\int_{X}(u-\psi)MA_{\omega}(\psi)=J_{\psi}(u)$. \newline
On the other hand, using the fact that $MA_{\omega}(\psi)=\mathbbm{1}_{\{\psi=0\}}\omega^{n}$ (\cite[Corollary 3.4]{DNT19}), we get
$$
0\leq \int_{X}(\psi-u)MA_{\omega}(\psi)\leq \int_{X}|u|MA_{\omega}(0)=||u||_{L^{1}}\leq C,
$$
where the second inequality is given by the weak compactness of $\{u\in PSH(X,\omega)\, : \, \sup_{X}=0\}$. The inequality $d(u,\psi)-C\leq J_{\psi}(u)$ clearly follows.
\end{proof}
Similarly to the case $\lambda<0$, since the $\psi$-relative energy is upper semicontinuous with respect to the weak topology (Proposition \ref{prop:Usc}), the continuity properties of $L_{\cdot,\lambda}$ are crucial for the variational approach of Theorems \ref{thmC}, \ref{thmD}. Therefore we recall the following well-known and important quantity (see \cite{DK99}).
\begin{defn}
\label{defn:CSE}
Let $u\in PSH(X,\omega)$. The quantity
$$
c(u):=\sup\Big\{p\geq 0\, : \, \int_{X}e^{-p u}\omega^{n}<+\infty\Big\}
$$
is called the \emph{complex singularity exponent} of $u$.
\end{defn}
Clearly $c(\cdot)$ increases when the singularities decreases, and, by the resolution of the strong openness conjecture, the supremum in the definition is never achieved, i.e. $e^{-c(u)u}\notin L^{1}$. Moreover the following important result holds.
\begin{thm}[\cite{DK99}, Main Theorem]
\label{thm:DK99}
The map $PSH(X,\omega)\ni u\to c(u)$ is lower semicontinuos. In particular, if $u_{k}\to u$ weakly and $c(u)>a$, then $e^{-au_{k}}\to e^{-au}$ converge in $L^{1}$.
\end{thm}
We also need to recall the definition of the \emph{Lelong numbers} and of the \emph{multiplier ideal sheaves}.\\
 Given $u\in PSH(X,\omega)$ and $x\in X$ the Lelong number $\nu(u,x)$ of $u$ at $x$ is given as
$$
\nu(u,x):=\sup\{\gamma\geq 0\, :\, u(z)\leq \gamma \log |z-x| +O(1) \, \mbox{on}\, U\}
$$
where $U\ni x$ is a holomorphic chart ($\nu(u,x)$ does not depend on the chart chosen). The Lelong number measures the logarithmic singularity of an $\omega$-psh function at a point $x$. \\
The multiplier ideal sheaf $\mathcal{I}(tu)$, $t\geq 0$, of $u\in PSH(X,\omega)$ is the analytic coherent sheaf whose germs are given as
$$
\mathcal{I}(tu,x):=\Big\{f\in\mathcal{O}_{X,x}\, : \, \int_{V}|f|^{2}e^{-tu}\omega^{n}<\infty\, \, \mbox{for some open set}\, \, x\in V\subset X \Big\}.
$$
\begin{prop}
\label{prop:Lelong}
Let $u\in PSH(X,\omega)$ and $\psi:=P_{\omega}[u]$. Then
\begin{equation}
\label{eqn:Show}
\nu(u,x)=\nu(\psi,x) \,\, \mathrm{and}\,\, \mathcal{I}(t u,x)=\mathcal{I}(t\psi,x)\, \, \mathrm{for}\, \mathrm{any} \,\, t>0, \, x\in X.
\end{equation}
In particular $c(u)=c(\psi)$ and $c(\cdot)$ is constant on any $\mathcal{E}(X,\omega,\psi)$ for $\psi\in\mathcal{M}^{+}$.
\end{prop}
\begin{proof}
We first observe that $c(u):=\sup\big\{p\geq 0\, : \, \mathcal{I}(pu)=\mathcal{O}_{X} \big\}$, thus (\ref{eqn:Show}) implies immediately $c(u)=c(\psi)$.  Moreover by \cite[Theorem 1.3]{DDNL17b}, letting $\psi\in\mathcal{M}^{+}$, $P_{\omega}[u]=\psi$ if and only if $u\in\mathcal{E}(X,\omega,\psi)$ as recalled in section \ref{sec:Pre}. The last assertion follows.\\
Next, we claim that $P_{\omega}[u](\psi)=\psi$. Indeed clearly $P_{\omega}[u](\psi)\leq \psi$, while, conversely, for any $C\in \mathbbm{R}_{\geq 0}$,
$$
P_{\omega}[u](\psi)\geq P_{\omega}\big(u+C,P_{\omega}[u]\big)\geq P_{\omega}\big(u+C,P_{\omega}(u+C,0)\big)=P_{\omega}(u+C,0),
$$
which implies $P_{\omega}[u](\psi)\geq P_{\omega}[u]=\psi$ as $P_{\omega}(u+C,0)\nearrow P_{\omega}[u]$.\\
Then the proof of (\ref{eqn:Show}) is similar to that of \cite[Theorem 1.1.(i)]{DDNL17a}, but we write the details as a courtesy to the reader.\\
Trivially $\gamma:=\nu(u,x)\geq\nu(\psi,x)$. Then $u(z)\leq \gamma\log|z|+O(1)$ locally in holomorphic coordinates centered at $x$ such that the unit ball $\mathbbm{B}\subset \mathbbm{C}^{n}$ is contained in the chart. Let $g$ be a smooth potential of $\omega$ such that $u+g,\psi+g\leq 0$ in $\mathbbm{B}$. Thus, locally
$$
g+\psi=g+P_{\omega}[u](\psi)\leq \sup\{v\in PSH(\mathbbm{B})\, :\, v\leq 0,v\leq \gamma \log|z|+O(1)\}
$$
where the inequality follows considering $P_{\omega}(u+C,\psi)$ for $C\to +\infty$ instead of $P_{\omega}[u](\psi)$ and noting that the right-hand is upper semicontinuous as it coincides with the pluricomplex Grenn function $G_{\mathbbm{B}}(z,0)$ of $\mathbbm{B}$ with a logarithmic pole at $0$ of order $\gamma$. Hence, \cite[Proposition 6.1]{Kli91} yields $\nu(\psi,x)\geq \gamma$ as $G_{\mathbbm{B}}(z,0)\sim \gamma\log|z| + O(1)$, and $\nu(u,x)=\nu(\psi,x)$ follows\\
For the second equality, letting $x\in X$ fixed, we observe that $\mathcal{I}(t\psi,x)=\mathcal{I}\big(tP_{\omega}(u+C,\psi),x\big)$ for $C\gg 0$ big enough as a consequence of the resolution of the strong openness conjecture (\cite[Theorem 1.1]{GZ14}) since $P_{\omega}(u+C,\psi)\nearrow \psi$ for $C\to +\infty$. Therefore to conclude the proof it is sufficient to note that $\mathcal{I}(tu,x)=\mathcal{I}\big(tP_{\omega}(u+C,\psi),x\big)$ for any $t,C>0$, $x\in X$ as $\psi$ is less singular than $u$.
\end{proof}
It is also possible to estimate the complex singularity exponent of $\psi$ in terms of the Lelong numbers by the following classical result.
\begin{prop}[\cite{Sko72}]
\label{prop:Alpha}
Let $\psi\in \mathcal{M}$ and set $\nu(\psi):=\sup_{x\in X}\nu(\psi,x)$.
Then
$$
\frac{2}{\nu(\psi)}\leq c(\psi)\leq \frac{2n}{\nu(\psi)},
$$
where we clearly mean $c(\psi)=+\infty$ if $\nu(\psi)=0$.
\end{prop}
We can now introduce an integrability condition which will be sufficient for the purposes of this paper.
\begin{defn}
\label{defn:SIC}
Given $\psi\in \mathcal{M}$, $\lambda>0$ and $p\in (1,\infty]$. We say that $[\psi]$ satisfies the \emph{Strong Integrability Condition} (SIC) with respect to $\lambda,p$ if
$$
c(\psi)>\frac{\lambda p}{p-1},
$$
where we mean $c(\psi)>\lambda$ if $p=\infty$.
\end{defn}
When $p=\infty$ the SIC $c(\psi)>\lambda$ is a necessary condition to solve the Monge-Ampère equation $MA_{\omega}(u)=e^{-\lambda u}f\omega^{n}$ in the class $\mathcal{E}(X,\omega,\psi)$. In general if $\psi\in\mathcal{M}^{+}$ then, as a consequence of Proposition \ref{prop:Lelong}, the SIC gives $e^{-\lambda u}f\in L^{1}$ for any $u\in\mathcal{E}(X,\omega,\psi)$ and $f\in L^{p}$ through a clear Hölder's pairing.
\begin{prop}
\label{lem:NNPP}
Let $u_{k},u\in PSH(X,\omega)$ such that $u_{k}\to u$ weakly, $\lambda>0$ and let $f_{k},f\in L^{p}$ for $p\in (1,\infty]$ non-negative functions such that $f_{k}\to f$ in $L^{p}$. Assume also that $\psi:=P_{\omega}[u]$ satisfies the SIC with respect to $\lambda, p$. Then
$$
e^{-\lambda u_{k}}f_{k}\to e^{-\lambda u}f
$$
in $L^{1}$ as $k\to \infty$.
\end{prop}
\begin{proof}
We set $g_{k}:=e^{-\lambda u_{k}}f_{k}$, $g:=e^{-\lambda u}f$ and $q:=p/(p-1)$ for the Sobolev conjugate of $p$. Since $\sup u_{k}\to \sup u$ (by Hartogs' Lemma, see \cite[Proposition 8.4]{GZ17}), we can also assume $\sup_{X}u_{k}\leq 0$ for any $k\in\mathbbm{N}$ without loss of generality. By the triangle inequality 
\begin{equation}
\label{eqn:NNPP}
||g_{k}-g||_{L^{1}}\leq ||e^{-\lambda u_{k}}(f_{k}-f)||_{L^{1}}+ ||(e^{-\lambda u_{k}}-e^{-\lambda u})f||_{L^{1}},
\end{equation}
and the strategy is to prove that both terms in the right-hand side converge to $0$ as $k\to \infty$. As immediate consequence of the Hölder's inequality we obtain
$$
||e^{-\lambda u_{k}}(f_{k}-f)||_{L^{1}}\leq ||e^{-u_{k}}||_{L^{\lambda q}}^{\lambda}||f_{k}-f||_{p},
$$
which converges to $0$ since $f_{k}\to f$ in $L^{p}$ by assumption while $||e^{-u_{k}}||_{L^{q}}$ is uniformly bounded for $k\gg 1$ big enough by Theorem \ref{thm:DK99} given the hypothesis $c(u)=c(\psi)>\lambda q$ (see also Proposition \ref{prop:Lelong}).\\
Regarding the second term, we still use Hölder's inequality to get
$$
||(e^{-\lambda u_{k}}-e^{-\lambda u})f||_{L^{1}}\leq ||e^{-\lambda u_{k}}-e^{-\lambda u}||_{L^{q}}||f||_{L^{p}}.
$$
As $f\in L^{p}$, it is sufficient to prove that $ ||e^{-\lambda u_{k}}-e^{-\lambda u}||_{L^{q}}\rightarrow 0$. However, since $\psi$ satisfied the SIC with respect to $\lambda, p$, combining Theorem \ref{thm:DK99} with Proposition \ref{prop:Lelong} we know that $e^{-\lambda q u_{k}}\to e^{-\lambda q u}$ in $L^{1}$, i.e. that $||e^{-\lambda u_{k}}||_{L^{q}}\to ||e^{-\lambda u}||_{L^{q}}$. Since by assumption we also have $e^{-\lambda u_{k}}\to e^{-\lambda u}$ almost everywhere, it follows that $e^{-\lambda u_{k}}\to e^{-\lambda u}$ in $L^{q}$, which concludes the proof.
\end{proof}
\begin{lem}
\label{lem:NNPPb}
Let $K\subset PSH(X,\omega)$ be a compact set and let $p>0$ such that $c(u)>p$ for any $u\in K$. Then there exists a constant $C=C_{K,p}$ such that
$$
\sup_{u\in K}\int_{X}e^{-p u}\omega^{n}\leq C.
$$
\end{lem}
\begin{proof}
Let us assume by contradiction that there exists a sequence $\{u_{j}\}_{j\in\mathbbm{N}}\subset K$ such that
\begin{equation}
\label{eqn:Finall}
\int_{X}e^{-p u_{j}}\omega^{n}\geq j
\end{equation}
for any $j\in\mathbbm{N}$. By compactness, up to considering a subsequence, we may also assume that $u_{j}\to u\in K$ weakly. In particular $\int_{X}e^{-p u}\omega^{n}<\infty$. Thus by Theorem \ref{thm:DK99} $e^{-pu_{k}}\to e^{-p u}$ in $L^{1}$, which contradicts (\ref{eqn:Finall}).
\end{proof}
\subsubsection{Proof of Theorem \ref{thmC}}
Together with the contraction property of Proposition \ref{prop:PropProie}, the key point to prove Theorem \ref{thmC} is to relate the coercivity of $F_{f,\psi,\lambda}$ to a \emph{Moser-Trudinger type of inequality} (see Proposition \ref{prop:MT} below).\newline

In this subsection we will use the following notation and assumptions:
\begin{itemize}
    \item by $\mathcal{P}(X)$ we will denote the set of all positive probability measures on $X$;
    \item $\mu:=f\omega^{n}$ for $0\leq f\in L^{p}\setminus \{0\}$, $p\in (1,+\infty]$;
    \item $\psi\in\mathcal{M}^{+}$ satisfies the SIC with respect to $\lambda, p$, where $\lambda>0$ is a fixed positive constant.
\end{itemize}
We need to recall the definition of \emph{entropy}.
\begin{defn}
Let $\nu_{1},\nu_{2}\in\mathcal{P}(X)$. The \emph{relative entropy} $H_{\nu_{1}}(\nu_{2})$ of $\nu_{2}$ with respect to $\nu_{1}$ is defined as follows. If $\nu_{2}$ is absolutely continuous with respect to $\nu_{1}$ with density $f=\frac{d\nu_{2}}{d\nu_{1}}$ satisfying $f\log f \in L^{1}(\nu_{1})$ then
$$
H_{\nu_{1}}(\nu_{2}):=\int_{X}f\log f d\nu_{1}=\int_{X}\log f d\nu_{2}.
$$
Otherwise $H_{\nu_{1}}(\nu_{2}):=+\infty$.
\end{defn}
We observe that probability measures $\nu$ with finite entropy with respect to $\mu$ have the following interesting property.
\begin{lem}
\label{lem:UsefulH}
Let $\nu\in\mathcal{P}(X)$ such that $H_{\mu}(\nu)<+\infty$. Then $PSH(X,\omega)\subset L^{1}(\nu)$ and $E_{\psi'}^{*}(\nu)<+\infty$, for any $\psi'\in \mathcal{M}^{+}$.
\end{lem}
\begin{proof}
Let $u\in PSH(X,\omega)$ and let $0<c<c(u)$. Assume also that $\sup_{X}u=0$ and let $a\in\mathbbm{R}$ such that $\mu_{cu}:=e^{-c(u+a)}\mu$ is a probability measure. Similarly, let $u_{k}:=\max(u,-k)$ and $a_{k}\in\mathbbm{R}$ such that $\mu_{cu_{k}}:=e^{-c(u_{k}+a_{k})}\mu\in \mathcal{P}(X)$.\newline
As $u_{k}$ is bounded, from the definition of entropy it is immediate to check that
\begin{equation}
    \label{eqn:M}
    0\leq H_{\mu_{cu_{k}}}(\nu)=H_{\mu}(\nu)+c\int_{X}(u_{k}+a_{k})d\nu\leq H_{\mu}(\nu)+c\int_{X}u_{k}d\nu+ca,
\end{equation}
where the last inequality follows observing that $a_{k}=\frac{\log \int_{X}e^{-cu_{k}}d\mu}{c}$ is an increasing sequence converging to $a$. Since $H_{\mu}(\nu)<+\infty$, (\ref{eqn:M}) clearly implies that
$$
\int_{X}(-u_{k})d\nu\leq C
$$
for a positive constant $C$ uniform in $k$. Hence, letting $k\to \infty$, by the Monotone Convergence Theorem ($u_{k}$ is decreasing and $\sup_{X}u_{k}$ is uniformly bounded by Hartogs' Lemma) we get
$$
\int_{X}(-u)d\nu\leq C,
$$
which concludes the first part of the proof.\newline
Next, we claim that there exists $\alpha>0$ such that $\sup_{\{u\in PSH(X,\omega)\}}\int_{X}e^{-\alpha(u-\sup_{X}u)}d\mu<+\infty$. Indeed, since $\mu=f\omega^{n}$ for $f\in L^{p}$, $p>1$, we can apply Hölder's inequality to reduce to the case $\mu=\omega^{n}$. Thus, the existence of such $\alpha$ follows combining Proposition \ref{prop:Alpha} and Lemma \ref{lem:NNPPb} with the fact that $\sup\{\nu(u,x)\,:\, u\in PSH(X,\omega),\, x\in X\}$ is finite (see \cite[Lemma 8.10]{GZ17}). In particular, letting $\psi'\in\mathcal{M}^{+}$, there exists a uniform constant $C$ such that for any $u\in \mathcal{E}^{1}(X,\omega,\psi')$ we have
$$
-\log\int_{X}e^{-\alpha u}d\mu\geq \alpha\sup_{X}u-C=\alpha\sup_{X}(u-\psi')-C\geq \frac{\alpha}{V_{\psi'}}E_{\psi'}(u)-C
$$
where the last inequality immediately follows from the definition of the Monge-Ampère energy. Thus, by Lemma \ref{lem:2.11} recalled below, we obtain
$$
H_{\mu}(\nu)\geq -\log\int_{X}e^{-\alpha u}d\mu-\alpha\int_{X}u d\nu\geq \frac{\alpha}{V_{\psi'}}\Big(E_{\psi'}(u)-\int_{X}(u-\psi')V_{\psi'}d\nu\Big)-C
$$
where we clearly also used $\psi'\leq 0$. Hence, taking the supremum over all $u\in\mathcal{E}^{1}(X,\omega,\psi')$, we get $H_{\nu}(\mu)\geq \frac{\alpha}{V_{\psi'}}E_{\psi'}^{*}(\nu)-C$, which concludes the proof.
\end{proof}
\begin{lem}[\cite{BBEGZ16}, Lemma $2.11$]
\label{lem:2.11}
For any lower semicontinuous function $g$ on $X$ and any $\nu_{1}\in\mathcal{P}(X)$,
$$
\log \int_{X}e^{g}d\nu_{1}=\sup_{\nu\in\mathcal{P}(X)}\Big(\int_{X}g d\nu_{2}-H_{\nu_{1}}(\nu_{2})\Big).
$$
\end{lem}
We can now introduce the functional $M_{f,\psi,\lambda}:\mathcal{E}^{1}(X,\omega,\psi)\to \overline{\mathbbm{R}}$ as
$$
M_{f,\psi,\lambda}(u)=E_{\psi}(u)-\frac{V_{\psi}}{\lambda}H_{\mu}\big(MA_{\omega}(u)/V_{\psi}\big)-\int_{X}u MA_{\omega}(u)
$$
if $H_{\mu}\big(MA_{\omega}(u)/V_{\psi}\big)<+\infty$, and $M_{f,\psi,\lambda}(u)=-\infty$ otherwise.\newline
Observe that by Lemma \ref{lem:UsefulH}, $M_{f,\psi,\lambda}$ is well-defined and takes finite values at any $u\in\mathcal{E}^{1}(X,\omega,\psi)$ whose associated Monge-Ampère probability measure has finite entropy with respect to $\mu=f\omega^{n}$. Moreover, since $E_{\psi}^{*}\big(MA_{\omega}(u/V_{\psi})\big)=E_{\psi}(u)-\int_{X}(u-\psi)MA_{\omega}(u)$ (by Theorem \ref{thm:A}), letting $F_{\psi}:\mathcal{P}(X)\to \overline{R}$, $F_{\psi}(\nu):=||\psi||_{L^{1}(\nu)}$,
$$
M_{f,\psi,\lambda}(u)=\Big(E_{\psi}^{*}-\frac{V_{\psi}}{\lambda}H_{\mu}+V_{\psi}F_{\psi}\Big)\big(MA_{\omega}(u)/V_{\psi}\big)
$$
for any $u\in\mathcal{E}^{1}(X,\omega,\psi)$ such that $H_{\mu}\big(MA_{\omega}(u)/V_{\psi}\big)<+\infty$. In particular $M_{f,\psi,\lambda}$ is translation invariant. \newline
This functional is clearly reminiscent of the Mabuchi functional (\cite{Mab86}) and we refer to the companion paper \cite{Tru20b} for more details in the Fano case.
\begin{prop}
\label{prop:DM}
$F_{f,\psi,\lambda}(u)\geq M_{f,\psi,\lambda}(u)$ for any $u\in\mathcal{E}^{1}(X,\omega,\psi)$, with equality if and only if $MA_{\omega}(u)=e^{-\lambda u+C}d\mu$ for a constant $C\in \mathbbm{R}$.
\end{prop}
\begin{proof}
Since we are assuming that $\psi\in\mathcal{M}^{+}$ satisfies the SIC with respect to $\lambda,p$, we have $F_{f,\psi,\lambda}(u)\in\mathbbm{R}$ (see also Proposition \ref{prop:Lelong}). Thus, we can suppose that $H_{\mu}\big(MA_{\omega}(u)/V_{\psi}\big)<+\infty$ (otherwise the inequality is trivial). By definition of the two functionals we clearly have
\begin{equation}
    \label{eqn:Z1}
    F_{f,\psi,\lambda}(u)-M_{f,\psi,\lambda}(u)=\frac{V_{\psi}}{\lambda}\log \int_{X}e^{-\lambda u}d\mu+\frac{V_{\psi}}{\lambda}H_{\mu}\big(MA_{\omega}(u)/V_{\psi}\big)+\int_{X}u MA_{\omega}(u).
\end{equation}
Next, setting $\mu_{u}:=\frac{e^{-\lambda u}\mu}{\int_{X}e^{-\lambda u}d\mu}\in\mathcal{P}(X)$ and letting $0\leq g\in L^{1}(\mu)$ such that $MA_{\omega}(u)/V_{\psi}=g\mu$, we observe that $MA_{\omega}(u)/V_{\psi}=g e^{\lambda u + \log \int_{X}e^{-\lambda u}d\mu} \mu_{u} $. Therefore, by definition of entropy we get
\begin{multline}
    \label{eqn:Z0}
    V_{\psi}H_{\mu_{u}}\big(MA_{\omega}(u)/V_{\psi}\big)=\int_{X}\log \big(g e^{\lambda u + \log \int_{X}e^{-\lambda u}d\mu}\big) MA_{\omega}(u)=\\
    =V_{\psi}H_{\mu}\big(MA_{\omega}(u)/V_{\psi}\big)+\lambda \int_{X}u MA_{\omega}(u)+V_{\psi}\log \int_{X}e^{-\lambda u}d\mu.
\end{multline}
Hence, combining (\ref{eqn:Z1}) and (\ref{eqn:Z0}), we obtain
$$
F_{f,\psi,\lambda}-M_{f,\psi,\lambda}=\frac{V_{\psi}}{\lambda}H_{\mu_{u}}\big(MA_{\omega}(u)/V_{\psi}\big),
$$
which concludes the proof since $H_{\mu_{u}}\big(MA_{\omega}(u)/V_{\psi}\big)\geq 0$ with equality if and only if $MA_{\omega}(u)=V_{\psi}\mu_{u}$ by \cite[Proposition 2.10.(ii)]{BBEGZ16}.
\end{proof}
In the following result we connect the coercivity of $F_{f,\psi,\lambda}$ (and of $M_{f,\psi,\lambda}$) to a Moser-Trudinger type of inequality.
\begin{prop}
\label{prop:MT}
Let $\psi\in\mathcal{M}^{+}$ and assume that $\psi$ satisfies the SIC with respect to $\lambda,p$, i.e. that $c(\psi)>\frac{\lambda p}{p-1}$. Then the followings are equivalent:
\begin{itemize}
    \item[i)] $F_{f,\psi,\lambda}$ is $d$-coercive over $\mathcal{E}^{1}_{norm}(X,\omega,\psi)$, i.e. there exist $A,B>0$ such that $F_{f,\psi,\lambda}(u)\leq -Ad(u,\psi)+B$ for any $u\in\mathcal{E}^{1}_{norm}(X,\omega,\psi)$;
    \item[ii)] $M_{f,\psi,\lambda}$ is $d$-coercive over $\mathcal{E}^{1}_{norm}(X,\omega,\psi)$;
    \item[iii)] there exist $q>1, C>0$ such that \begin{equation}
    \label{eqn:Moser}
    ||e^{\lambda(\psi-u)}||_{L^{q}(e^{-\lambda\psi}\mu)}\leq Ce^{-\frac{\lambda}{V_{\psi}}E_{\psi}(u)}
    \end{equation}
    for any $u\in\mathcal{E}^{1}(X,\omega,\psi)$.
\end{itemize}
\end{prop}
\begin{rem}
\label{rem:ConstantA}
\emph{It is easy to check that the constant $A>0$ in the $d$-coercivity of $F_{f,\psi,\lambda}$ cannot be larger than $1$. Indeed it easily follows from $(A-1)E_{\psi}(u) +B\geq \frac{V_{\psi}}{\lambda}\log\int_{X}e^{-\lambda u}f\omega^{n}\geq \frac{V_{\psi}}{\lambda}\log||f||_{L^{1}}$ for any $u\in\mathcal{E}^{1}_{norm}(X,\omega,\psi)$ and the fact that $\inf_{u\in\mathcal{E}^{1}_{norm}(X,\omega,\psi)}E_{\psi}(u)=-\infty$ for any $\psi\in\mathcal{M}^{+}$.}
\end{rem}
\begin{proof}
The implication $(i)\Rightarrow (ii)$ immediately follows from Proposition \ref{prop:DM}. Then let assume $(ii)$ to hold, i.e. that there exists $A>0,B\geq 0$ such that
$$
M_{f,\psi,\lambda}(u)\leq -Ad(u,\psi)+B
$$
for any $u\in\mathcal{E}^{1}_{norm}(X,\omega,\psi)$. We recall that (see \cite[Lemma 3.1.(i)]{Tru20})
$$
J_{\psi}(u)\geq \frac{1}{n+1}\int_{X}(\psi-u)\big(MA_{\omega}(u)-MA_{\omega}(\psi)\big),
$$
which implies
\begin{multline}
\label{eqn:ZZ}
    J_{\psi}(u)=\frac{n+1}{n}J_{\psi}(u)-\frac{1}{n}J_{\psi}(u)\geq \frac{1}{n}\int_{X}(\psi-u)\big(MA_{\omega}(u)-MA_{\omega}(\psi)\big)-\frac{1}{n}J_{\psi}(u)=\\
    =\frac{1}{n}\Big(E_{\psi}(u)-\int_{X}(u-\psi)MA_{\omega}(u)\Big)=\frac{1}{n}E_{\psi}^{*}\big(MA_{\omega}(u)/V_{\psi}\big),
\end{multline}
where in the last equality we used Theorem \ref{thm:A}. Therefore, as $d(u-\sup_{X}u,\psi)\geq J_{\psi}(u-\sup_{X}u)=J_{\psi}(u) $ for any $u\in\mathcal{E}^{1}(X,\omega,\psi)$ (Lemma \ref{lem:J}) and $M_{f,\psi,\lambda}$ is translation invariant, we get
\begin{equation}
\label{eqn:Mabu}
M_{f,\psi,\lambda}(u)\leq -AJ_{\psi}(u)+B\leq -\frac{A}{n}E^{*}_{\psi}\big(MA_{\omega}(u)/V_{\psi}\big)+B
\end{equation}
for any $u\in\mathcal{E}^{1}(X,\omega,\psi)$, where we used (\ref{eqn:ZZ}) in the last inequality. By definition of $M_{f,\psi,\lambda}$, (\ref{eqn:Mabu}) is equivalent to
\begin{equation}
    \label{eqn:ZZZ}
    \frac{V_{\psi}}{\lambda}H_{\mu}\big(MA_{\omega}(u)/V_{\psi}\big)+\int_{X}\psi MA_{\omega}(u)\geq q E_{\psi}^{*}\big(MA_{\omega}(u)/V_{\psi}\big)-B
\end{equation}
for $q:=1+A/n>1$, which implies $\frac{V_{\psi}}{\lambda}H_{\mu}(\nu)+\int_{X}\psi\, V_{\psi}d\nu\geq qE_{\psi}^{*}(\nu)-B$ for any $\nu\in \mathcal{P}(X)$ such that $H_{\mu}(\nu)<+\infty$ since probability measures with finite entropy with respect to $\mu$ have finite energy by Lemma \ref{lem:UsefulH} and hence they belong to the image of the Monge-Ampère operator $MA_{\omega}/V_{\psi}:\mathcal{E}^{1}(X,\omega,\psi)\to \mathcal{P}(X)$ (by Theorem \ref{thm:A}).\newline
Next, suppose that $u$ has $\psi$-relative minimal singularities. In particular $\int_{X}e^{q\lambda(\psi-u)-\lambda \psi}d\mu<+\infty$, as $\mu=f\omega^{n}$ for $f\in L^{p}$ and $\psi$ satisfies the SIC with respect to $\lambda,p$. Letting $a\in \mathbbm{R}$ such that $e^{(q-1)\lambda\psi+a}\mu\in \mathcal{P}(X)$, by Lemma \ref{lem:2.11} it follows that
$$
+\infty>\log\int_{X}e^{q\lambda(\psi-u)-\lambda\psi+a}d\mu=\sup_{\nu\in \mathcal{P}(X)}\Big(-q\lambda\int_{X}ud\nu-H_{e^{(q-1)\lambda\psi+a}\mu}(\nu)\Big).
$$
Thus, for any $\epsilon>0$ fixed there exists $\nu_{\epsilon}\in \mathcal{P}(X)$ such that
$H_{e^{(q-1)\lambda\psi+a}\mu}(\nu_{\epsilon})<+\infty$ and
$$
\log\int_{X}e^{q\lambda(\psi-u)}e^{-\lambda\psi+a}d\mu\leq \epsilon-q\lambda\int_{X}u\,d\nu_{\epsilon}-H_{e^{(q-1)\lambda\psi+a}\mu}(\nu_{\epsilon}).
$$
Moreover, by the definition of entropy it is immediate to check that $H_{\mu}(\nu_{\epsilon})=H_{e^{(q-1)\lambda \psi+a}}(\nu_{\epsilon})+(q-1)\lambda \int_{X}\psi d\nu_{\epsilon} +a\leq H_{e^{(q-1)\lambda \psi+a}}(\nu_{\epsilon})+a$, which in particular yields $H_{\mu}(\nu_{\epsilon})<+\infty$. Therefore by an easy calculation, we obtain
\begin{multline}
    \label{eqn:ZZZZ}
    \frac{V_{\psi}}{\lambda}\log\int_{X}e^{q\lambda(\psi-u)}e^{-\lambda\psi+a}d\mu\leq \frac{V_{\psi}}{\lambda}\epsilon +q\int_{X}(\psi-u)V_{\psi}d\nu_{\epsilon}+\frac{V_{\psi}}{\lambda}a-\frac{V_{\psi}}{\lambda}H_{\mu}(\nu_{\epsilon})-\int_{X}\psi V_{\psi}d\nu_{\epsilon}\leq\\
    \leq \frac{V_{\psi}}{\lambda}(\epsilon+a)+B+q\Big(\int_{X}(\psi-u)V_{\psi}d\nu_{\epsilon}-E_{\psi}^{*}(\nu_{\epsilon})\Big)
\end{multline}
where in the last inequality we used (\ref{eqn:ZZZ}). Next, we observe that by \cite[Proposition 2.12.(iii)]{Tru19}, for any $v\in\mathcal{E}^{1}(X,\omega,\psi)$
\begin{multline*}
    E_{\psi}(u)\leq E_{\psi}(v)+\int_{X}(u-v)MA_{\omega}(v)=E_{\psi}(v)+\int_{X}(\psi-v)MA_{\omega}(v)+\int_{X}(u-\psi)MA_{\omega}(v)=\\
    =E_{\psi}^{*}\big(MA_{\omega}(v)/V_{\psi}\big)+\int_{X}(u-\psi)MA_{\omega}(v),
\end{multline*}
which for $v$ such that $MA_{\omega}(v)=V_{\psi}\nu_{\epsilon}$ gives
\begin{equation}
    \label{eqn:Y}
    \int_{X}(\psi-u)V_{\psi}\nu_{\epsilon}-E_{\psi}^{*}(\nu_{\epsilon})\leq -E_{\psi}(u). 
\end{equation}
Hence, combining (\ref{eqn:ZZZZ}) with (\ref{eqn:Y}) and letting $\epsilon\to 0$, we get
$$
\frac{V_{\psi}}{\lambda}\log\int_{X}e^{q\lambda(\psi-u)}e^{-\lambda\psi}d\mu\leq-qE_{\psi}(u)+B
$$
for any $u\in\mathcal{E}^{1}(X,\omega,\psi)$ with $\psi$-relative minimal singularities, which is equivalent to (\ref{eqn:Moser}) setting $C:=e^{\frac{\lambda B}{q V_{\psi}}}$. By continuity of the Monge-Ampère energy along decreasing sequences and by Monotone Convergence Theorem, we can extend the same inequality to all $\mathcal{E}^{1}(X,\omega,\psi)$ and getting $(iii)$. Indeed we can approximate $u\in\mathcal{E}^{1}(X,\omega,\psi)$ by the sequence of elements with $\psi$-relative minimal singularities $u_{k}:=\max(\psi-k,u)$.\newline
Finally, assuming $(iii)$, it remains to prove the $d$-coercivity of $F_{f,\psi,\lambda}$. Fix $\epsilon\in (0,1)$ small enough such that $q:=1+\epsilon$ satisfies (\ref{eqn:Moser}). Then for any $u\in\mathcal{E}^{1}_{norm}(X,\omega,\psi)$, combining the equality $(u-\psi)=(1+\epsilon)(1-\epsilon)(u-\psi)+\epsilon^{2}(u-\psi)$ with the convexity of $f\to \log \int_{X}e^{-f}d\nu$, we get
$$
\log \int_{X}e^{-\lambda(u-\psi)}e^{-\lambda\psi}d\mu\leq (1-\epsilon)\log\int_{X}e^{-\lambda(1+\epsilon)(u-\psi)}e^{-\lambda\psi}d\mu+\epsilon\log\int_{X}e^{-\lambda\epsilon(u-\psi)}e^{-\lambda\psi}d\mu,
$$
and the first term in the right-hand side is dominated by $(1-\epsilon)\big(-\frac{(1+\epsilon)\lambda}{V_{\psi}}E_{\psi}(u)+D\big)$ for a constant $D$ by the hypothesis $(iii)$. For the second term,
\begin{equation}
\label{eqn:Feb2021}
\int_{X}e^{-\epsilon\lambda (u-\psi)}e^{-\lambda\psi}d\mu\leq\int_{X}e^{-\lambda\epsilon u}e^{-\lambda\psi}d\mu,
\end{equation}
and we claim that $\int_{X}e^{-\lambda \epsilon u}e^{-\lambda \psi}d\mu$ is uniformly bounded if $\epsilon\ll 1 $ is small enough. Indeed, since by the SIC $c(\psi)>\frac{\lambda p}{p-1}$, applying Hölder's inequality it is clearly enough to prove that there exists $\delta>0$ very small such that
$$
\sup_{u\in\mathcal{E}^{1}_{norm}(X,\omega,\psi)}\int_{X}e^{-\delta u}\omega^{n}<+\infty.
$$
But this follows from the uniform version of the Skoda's Integrability Theorem (see \cite[Corollary 3.2]{Zer01}) since $\mathcal{E}^{1}_{norm}(X,\omega,\psi)\subset \{u\in PSH(X,\omega)\, : \, \sup_{X}u=0\}$ is a compact set.\newline
Therefore it follows that
$$
\frac{V_{\psi}}{\lambda}\log \int_{X}e^{-\lambda u}d\mu=\frac{V_{\psi}}{\lambda}\log\int_{X}e^{-\lambda(u-\psi)}e^{-\lambda\psi}d\mu\leq -(1-\epsilon^{2})E_{\psi}(u)+B
$$
for a constant $B$. Hence
$$
F_{f,\psi,\lambda}(u)\leq E_{\psi}(u)-(1-\epsilon^{2})E_{\psi}(u)+B=-\epsilon^{2}d(\psi,u)+B,
$$
for any $u\in\mathcal{E}^{1}_{norm}(X,\omega,\psi)$, which concludes the proof.
\end{proof}
We can now prove Theorem \ref{thmC}, which as said in the Introduction represents an openness result for a new continuity method with movable singularities.
\begin{reptheorem}{thmC}
Let $\psi\in\mathcal{M}^{+}$, $\lambda>0$ and let $0\leq f\in L^{p}\setminus\{0\}$ for $p\in(1,\infty]$. Assume also that $\psi$ satisfies the SIC with respect to $\lambda, p$, i.e. that $c(\psi)>\frac{\lambda p}{p-1}$. If $F_{\psi}:=F_{f,\psi,\lambda}$ is $d$-coercive over $\mathcal{E}^{1}_{norm}(X,\omega,\psi)$, then there exists $A>1$ such that $F_{\psi'}:=F_{f,\psi',\lambda}$ is $d$-coercive over $\mathcal{E}^{1}(X,\omega,\psi')$ for any $\psi'$ less singular than $\psi$ such that $V_{\psi'}< AV_{\psi}$.
In particular the complex Monge-Ampère equation
\begin{equation}
\label{eqn:ENNMAD}
\begin{cases}
MA_{\omega}(v)=e^{-\lambda v}f\omega^{n}\\
v\in\mathcal{E}^{1}(X,\omega,\psi')
\end{cases}
\end{equation}
admits a solution for any $\psi'$ less singular than $\psi$ such that $V_{\psi'}<AV_{\psi}$.
\end{reptheorem}
\begin{proof}
We divide the proof in two parts. We first prove that the $d$-coercivity of $F_{\psi'}$ over $\mathcal{E}^{1}(X,\omega,\psi')$ implies the existence of a solution of (\ref{eqn:ENNMAD}) for a fixed $\psi'\succcurlyeq \psi$, then we show that the $d$-coercivity of $F_{\psi}$ yields a constant $A>1$ such that $F_{\psi'}$ is $d$-coercivity for any $\psi'$ less singular than $\psi$ such that $V_{\psi'}<AV_{\psi}$.\newline
Let $\psi'\succcurlyeq \psi$ and assume that $F_{\psi'}$ is $d$-coercive over $\mathcal{E}^{1}_{norm}(X,\omega,\psi')$ with respect to constants $A>0,B\geq 0$. Observe also that $\psi'$ satisfies the SIC with respect to $\lambda, p$. Then letting $\{u_{k}\}_{k\in\mathbbm{N}}\subset \mathcal{E}^{1}_{norm}(X,\omega,\psi')$ be a maximizing sequence for $F_{\psi'}$, i.e. $F_{\psi'}(u_{k})\nearrow \sup_{\mathcal{E}^{1}_{norm}(X,\omega,\psi')}F_{\psi'}$, by the coercivity we immediately have
$$
d(\psi',u_{k})\leq D
$$
for a constant $D\in\mathbbm{R}_{\geq 0}$. The latter condition means that $E_{\psi'}(u_{k})\geq -C$ for some positive constant $C>0$ uniform in $k$. Therefore the compactness of Proposition \ref{prop:Usc} implies that, up to considering a subsequence, $u_{k}\to u\in \mathcal{E}^{1}_{norm}(X,\omega,\psi')$ weakly. Thus Proposition \ref{lem:NNPP} and again Proposition \ref{prop:Usc} give
$$
\sup_{\mathcal{E}^{1}_{norm}(X,\omega,\psi)}F_{\psi'}=\lim_{k\to \infty}F_{\psi'}(u_{k})\leq F_{\psi'}(u),
$$
i.e. $u$ is a maximizer of $F_{\psi'}$ over $\mathcal{E}^{1}_{norm}(X,\omega,\psi')$. Hence since $F_{\psi'}$ is translation invariant, by Theorem \ref{thm:Max} there exists a constant $B$ such that $u+B$ solves (\ref{eqn:ENNMAD}) which concludes the first part of the proof.\newline
Next, by Proposition \ref{prop:MT} the $d$-coercivity of $F_{\psi}$ implies that there exists $q>1,C>0$ such that
\begin{equation}
    \label{eqn:MT1}
    ||e^{\lambda(\psi-u)}||_{L^{q}(e^{-\lambda\psi}\mu)}\leq Ce^{-\frac{\lambda}{V_{\psi}}E_{\psi}(u)}
\end{equation}
for any $u\in\mathcal{E}^{1}(X,\omega,\psi)$. Moreover by a simple application of Hölder's inequality the inequality (\ref{eqn:MT1}) holds for any $q'\in(1,q]$. Thus, for $0<\epsilon<1$ small enough to be fixed later, we can assume that for $q=1+\epsilon$ the inequality (\ref{eqn:MT1}) holds. Let also $\psi'\succcurlyeq \psi$, $u'\in\mathcal{E}^{1}(X,\omega,\psi')$ and $u:=P_{\omega}[\psi](u')$ where $u'$ is chosen such that $\sup_{X}u'=0$. Then, similarly to the proof of Proposition \ref{prop:MT}, by the convexity of $f\to \log \int_{X}e^{-f}d\nu$ we have
\begin{equation}
    \label{eqn:MTZ1}
    \log \int_{X}e^{-\lambda u'}d\mu=\log \int_{X}e^{\lambda(\psi-u')}e^{-\lambda\psi}d\mu\leq (1-\epsilon)\log\int_{X}e^{\lambda(1+\epsilon)(\psi-u')}e^{-\lambda\psi}d\mu+\epsilon\log\int_{X}e^{\lambda\epsilon(\psi-u')}e^{-\lambda\psi}d\mu
\end{equation}
since clearly $\psi-u'=(1-\epsilon)(1+\epsilon)(\psi-u')+\epsilon^{2}(\psi-u')$.
As $u'\geq u$ and $q=1+\epsilon$, for the first term in (\ref{eqn:MTZ1}) we have
\begin{multline*}
    (1-\epsilon)\log\int_{X}e^{\lambda(1+\epsilon)(\psi-u')}e^{-\lambda\psi}d\mu\leq (1-\epsilon^{2})\log\Big( \int_{X}e^{\lambda (1+\epsilon)(\psi-u)}e^{-\lambda \psi}d\mu\Big)^{1/(1+\epsilon)}=\\
    =(1-\epsilon^{2})\log\lVert e^{\lambda(\psi-u)}\rVert_{L^{q}(e^{-\lambda\psi}\mu)}\leq -(1-\epsilon^{2})\frac{\lambda}{V_{\psi}}E_{\psi}(u)+C_{1}
\end{multline*}
for a constant $C_{1}$, where in the last inequality we clearly used (\ref{eqn:MT1}). Regarding the second term in (\ref{eqn:MTZ1}), applying Hölder's inequality, as $\psi$ satisfying the SIC with respect to $\lambda,p$, it is easy to see that there exist $q'>1, C_{2}>0$ such that
$$
\int_{X}e^{\lambda\epsilon(\psi-u')}e^{-\lambda\psi}d\mu\leq C_{2} \Big(\int_{X}e^{-q'\epsilon\lambda u'}\omega^{n}\Big)^{1/q'}.
$$
Moreover, we recall that there exists $\alpha>0$ such that $\int_{X}e^{-\alpha v}\omega^{n}$ is uniformly bounded varying $v\in \{u\in PSH(X,\omega)\, : \, \sup_{X}u=0\}$. Indeed, as already underlined during the proof of Lemma \ref{lem:UsefulH}, this well-known fact follows combining Proposition \ref{prop:Alpha}, Lemma \ref{lem:NNPPb} and \cite[Lemma 8.10]{GZ17}. Thus if $\epsilon>0$ is small enough, the second term in (\ref{eqn:MTZ1}) is uniformly bounded.\newline
Summarizing, from (\ref{eqn:MTZ1}) we obtain
$$
\log\int_{X}e^{-\lambda u'}d\mu\leq -(1-\epsilon^{2})\frac{\lambda}{V_{\psi}}E_{\psi}(u)+C_{3}\leq -(1-\epsilon^{2})\frac{\lambda}{V_{\psi}}E_{\psi'}(u')+C_{3}
$$
for a constant $C_{3}$, where the last inequality follows from Proposition \ref{prop:PropProie}. Indeed the latter yields $P_{\omega}[\psi](\psi')=P_{\omega}[\psi]\big(P_{\omega}[\psi'](0)\big)=P_{\omega}[\psi](0)=\psi$ ($\psi$ and $\psi'$ are model type envelopes) and $-E_{\psi}(u)=d(\psi,u)=d\big(P_{\omega}[\psi](\psi'),P_{\omega}[\psi](u')\big)\leq d(\psi',u')=-E_{\psi'}(u')$. Therefore,
\begin{multline*}
    F_{\psi'}(u')=E_{\psi'}(u')+\frac{V_{\psi'}}{\lambda}\log\int_{X}e^{-\lambda u'}d\mu\leq \Big(1-(1-\epsilon^{2})\frac{V_{\psi'}}{V_{\psi}}\Big)E_{\psi'}(u')+C_{4}=-\Big(1-(1-\epsilon^{2})\frac{V_{\psi'}}{V_{\psi}}\Big)d(\psi',u')+C_{4},
\end{multline*}
which is the $d-$coercivity requested if 
$$
V_{\psi'}< \frac{V_{\psi}}{1-\epsilon^{2}}.
$$
Hence, setting $A:=\frac{1}{1-\epsilon^{2}}>1$ concludes the proof.

\end{proof}
\begin{cor}
\label{cor:Open}
Let $0\leq f\in L^{p}\setminus \{0\}$, $p>1$ and let $[0,1]\ni t\to \psi_{t}\in\mathcal{M}^{+}$ be a increasing continuous path such that $\psi_{0}$ satisfies the SIC with respect to $\lambda,p$. Then the set
$$
S:=\big\{t\in[0,1]\, : \, F_{f,\psi_{t},\lambda}\mbox{ is }d\mbox{-coercive over } \mathcal{E}^{1}_{norm}(X,\omega,\psi_{t})\big\}
$$
is open with respect to the topology $\mathcal{T}$ generated by $\{[a,b)\}_{0\leq a\leq b\leq 1}$.
\end{cor}
\begin{proof}
Since the path $t\to \psi_{t}$ is continuous by hypothesis, the continuity of $t\to V_{t}$ follows by what said in section \ref{sec:Pre}. Therefore the openness required is a consequence of Theorem \ref{thmC}.
\end{proof}
\subsubsection{Proof of Theorem \ref{thmD}}
Finally we give necessary conditions for the strong continuity of a sequence of solutions of $MA_{\omega}(u_{k})=e^{-\lambda u_{k}}f_{k}\omega^{n}$ with prescribed singularities, i.e. Theorem \ref{thmD}.
\begin{reptheorem}{thmD}
Let $\lambda>0$, $\{\psi_{k}\}_{k\in \mathbbm{N}}\subset \mathcal{M}^{+}$ be a totally ordered sequence converging to $\psi\in\mathcal{M}^{+}$, and $f_{k},f\geq 0$ non trivial such that $f_{k}\to f$ in $L^{p}$ as $k\to \infty$ for $p\in (1,\infty]$. Assume also the following conditions:
\begin{itemize}
\item[(i)] $\psi$ satisfies the SIC with respect to $\lambda,p$, i.e. $c(\psi)>\frac{\lambda p}{p-1}$;
\item[(ii)] the complex Monge-Ampère equations 
$$
\begin{cases}
MA_{\omega}(u_{k})=e^{-\lambda u_{k}}f_{k}\omega^{n}\\
u_{k}\in \mathcal{E}^{1}(X,\omega,\psi_{k});
\end{cases}
$$
admit solutions $u_{k}$ given as maximizers of $F_{f_{k},\psi_{k},\lambda}$;
\item[(iii)] $\sup_{X} u_{k}\leq C$ for a uniform constant $C$.
\end{itemize}
Then there exists a subsequence $\{u_{k_{h}}\}_{h\in\mathbbm{N}}$ that converges strongly to $u\in\mathcal{E}^{1}(X,\omega,\psi)$ solution of
$$
\begin{cases}
MA_{\omega}(u)=e^{-\lambda u}f\omega^{n}\\
u\in\mathcal{E}^{1}(X,\omega,\psi).
\end{cases}
$$
\end{reptheorem}
\begin{proof}
Without loss of generality we may assume that $\{\psi_{k}\}_{k\in\mathbbm{N}}$ is monotone.\newline
We set $F_{k}:=F_{f_{k},\psi_{k},\lambda}$ for any $k\in\mathbbm{N}$, $F:=F_{f,\psi,\lambda}$ and $v_{k}:=u_{k}-\sup_{X}u_{k}\in\mathcal{E}^{1}_{norm}(X,\omega,\psi_{k})$. In particular, $MA_{\omega}(v_{k})=e^{-\lambda (v_{k}+\sup_{X}u_{k})}f_{k}\omega^{n}$ for any $k\in\mathbbm{N}$ and, up to considering a subsequence, we may assume that $v_{k}$ converges weakly to a function $v\in PSH(X,\omega)$. Moreover, using the key assumption $(iii)$ and the clear equality $E_{\psi_{k}}(\psi_{k})=0$, we obtain
\begin{multline*}
C_{1}\leq \frac{V_{\psi_{k}}}{\lambda}\log ||f_{k}||_{L^{1}}\leq \frac{V_{\psi_{k}}}{\lambda}\log \int_{X}e^{-\lambda \psi_{k}}f_{k}\omega^{n}=F_{k}(\psi_{k})\leq\\
\leq F_{k}(v_{k})=E_{\psi_{k}}(v_{k})+\frac{V_{\psi_{k}}}{\lambda}\log\int_{X}e^{-\lambda v_{k}}f_{k}\omega^{n}=E_{\psi_{k}}(v_{k})+\frac{V_{\psi_{k}}}{\lambda}\log\int_{X}e^{-\lambda u_{k}}f_{k}\omega^{n}+V_{\psi_{k}}\sup_{X}u_{k}=\\
=E_{\psi_{k}}(v_{k})+\frac{V_{\psi_{k}}}{\lambda}\log\int_{X}MA_{\omega}(u_{k})+V_{\psi_{k}}\sup_{X}u_{k}=E_{\psi_{k}}(v_{k})+\frac{V_{\psi_{k}}}{\lambda}\log V_{\psi_{k}}+V_{\psi_{k}}\sup_{X}u_{k}\leq E_{\psi_{k}}(v_{k})+C_{2}.
\end{multline*}
for two uniform constants $C_{1},C_{2}$. Note that we also used $v_{k}=u_{k}-\sup_{X}u_{k}$ and $MA_{\omega}(u_{k})=e^{-\lambda u_{k}}f_{k}\omega^{n}$. Therefore by Proposition \ref{prop:Usc} we deduce $v\in\mathcal{E}^{1}(X,\omega,\psi)$ and $|\sup_{X}u_{k}|\leq C_{3}$ uniformly, which in turn yield $u_{k}\to u$ weakly for $u\in\mathcal{E}^{1}(X,\omega,\psi)$ and $\limsup_{k\to \infty}E_{\psi_{k}}(u_{k})\leq E_{\psi}(u)$. Moreover, as by Proposition \ref{lem:NNPP} $\int_{X}e^{-\lambda u_{k}}f_{k}\omega^{n}\to \int_{X}e^{-\lambda u}f\omega^{n}$, it follows that
$$
\limsup_{k\to \infty} F_{k}(u_{k})\leq F(u).
$$
On the other hand similarly to the proof of Theorem \ref{thmB}, letting $\varphi\in\mathcal{H}$ we obtain
$$
\liminf_{k\to \infty} F_{k}(u_{k})\geq \liminf_{k\to \infty}F_{k}\big(P_{\omega}[\psi_{k}](\varphi)\big)=F\big(P_{\omega}[\psi](\varphi)\big)
$$
combining Lemma \ref{lem:KeyConv} and Proposition \ref{lem:NNPP}, which in turn, together with \cite[Theorem 1]{BK07} and with the continuity of $F$ along decreasing sequences, implies
$$
\liminf_{k\to \infty}F_{k}(u_{k})\geq \sup_{\mathcal{E}^{1}(X,\omega,\psi)}F.
$$
Hence $u$ is a maximizer of $F$ over $\mathcal{E}^{1}(X,\omega,\psi)$ and $F_{k}(u_{k})\to F(u)$. In particular, $u_{k}\to u$ strongly and there exists a constant $C\in\mathbbm{R}$ such that $MA_{\omega}(u)=e^{-\lambda (u+C)} f\omega^{n}$ (Theorem \ref{thm:Max}). However
$$
e^{-\lambda C}\int_{X}e^{-\lambda u}f\omega^{n}=V_{\psi}=\lim_{k\to \infty}V_{k}=\lim_{k\to \infty}\int_{X}e^{-\lambda u_{k}}f_{k}\omega^{n}=\int_{X}e^{-\lambda u}f\omega^{n},
$$
which yields $C=0$ and concludes the proof.
\end{proof}
\begin{rem}
\emph{Observe that the assumption $(i)$ in Theorem \ref{thmD} is satisfied if all the Lelong numbers of $\psi_{k}$ are small enough (Proposition \ref{prop:Alpha}), while $(ii)$ is a natural hypothesis when all the solutions are given as maximizers (see also \cite[Theorem C]{Tru20b}). As stated in the Introduction the real big obstacle is the bound in $(iii)$, which is necessary even when $f_{k}\equiv f$ (Example \ref{esem:Necess} below, see also \cite{Tru20b} for a deeper discussion regarding $(iii)$ in the Fano case).}
\end{rem}
\section{Log semi-Kähler Einstein metrics with prescribed singularities.}
\label{sec:KE}
On a line bundle $L\to X$ any (smooth) hermitian metric $h$ can be described by its \emph{weight} $\phi=\{\phi_{\alpha}\}_{\alpha\in I}$ defined locally for a trivializing local section $s_{\alpha}$ of $L$ on a open set $U_{\alpha}$ as $\phi_{\alpha}:=-\log |s_{\alpha}|^{2}_{h}$. Observe that the current $dd^{c}\phi$ is globally well-defined and represents the curvature of $h$. In this section we identify the hermitian metrics with their weights, and we say \emph{metric} for simplicity.\\ 
Given a $\mathbbm{Q}$-divisor $D$ on $X$ we have the following key definition.
\begin{defn}[\cite{BBEGZ16}, Definition 3.1]
\label{defn:Adap}
Let $\phi$ be a metric on $-r(K_{X}+D)$ where $r\in\mathbbm{N}$ such that $rD$ is a divisor. The \emph{adapted measure} $\mu_{\phi}$ is locally defined by choosing a nowhere zero section $\sigma$ of $r(K_{X}+D)$ over a small open set $U$ and setting
$$
\mu_{\phi}:=(i^{rn^{2}}\sigma \wedge \bar{\sigma})^{1/r}/|\sigma|_{\phi}^{2/r}.
$$
\end{defn}
We observe that $\mu_{\phi}$ is globally defined since the definition does not depend on the choice of $\sigma$. Moreover $\mu_{\phi_{1}}=\mu_{\phi_{2}}$ if $\phi_{i}$ are metric on $-r_{i}(K_{X}+D)$ such that $r_{2}\phi_{1}=r_{1}\phi_{2}$. This property allows to enlarge the definition of adapted measures to $\mathbbm{Q}$-line bundles where $\phi$ is a metric on $-(K_{X}+D)$ if there exists $r\in\mathbbm{N}$ divisible enough such that $r\phi$ is a metric on $-r(K_{X}+D)$. \\
Note that if $D=0$ and $\phi$ is a metric on $-K_{X}$, then locally
$$
\mu_{\phi}=e^{-\phi} i^{n^{2}}dz_{1}\wedge \cdots \wedge dz_{n}\wedge d\bar{z}_{1}\wedge \cdots\wedge d\bar{z}_{n}.
$$
More generally, by the natural identification of $-(K_{X}+D)$ with $-K_{X}$ on the complement of the support of the divisor $D$, if $\phi$ is a metric on $-(K_{X}+D)$ then locally on $X\setminus \mbox{Supp}(D)$
$$
\mu_{\phi}=e^{-(\phi+\frac{2}{r}\log|s_{rD_{+}}|-\frac{2}{r}\log|s_{rD_{-}}|)}i^{n^{2}}\Omega\wedge \bar{\Omega}
$$
for $s_{rD_{+}},s_{rD_{-}}$ holomorphic sections cutting respectively the effective divisors $rD_{+},rD_{-}$ for $r\in \mathbbm{N}$ where $D=D_{+}-D_{-}$, and $\Omega$ is a nowhere zero local holomorphic section of $K_{X}$ (see also \cite{BBJ15}). Furthermore the adapted measures are compatible with respect to blow-ups of smooth centers. Indeed if $p:Y\to X$ is a morphism given by a sequence of blow-ups of smooth centers, letting $D'$ such that $p^{*}(K_{X}+D)=K_{Y}+D'$, $\mu_{p^{*}\phi}$ coincides with the \emph{lift} of $\mu_{\phi}$ (usually denoted by $\tilde{\mu}_{\phi}$), i.e. with the trivial extension of the push-forward by $p^{-1}$ of $\mu_{\phi}$ over the Zariski open set where $p$ is an isomorphism. Conversely, $p_{*}\mu_{p^{*}\phi}=\mu_{\phi}$.\\

Next, it is well-known that smooth positive volume forms $\mu$ are in one-one correspondence with metrics on the canonical line bundle $K_{X}$ and the relationship is given by
\begin{equation}
\label{eqn:Ricci}
\mu=e^{-f}i^{n^{2}}\Omega\wedge \bar{\Omega}
\end{equation}
where $f:=2\log|\Omega|_{\phi}$ for any nowhere zero local holomorphic section $\Omega$ of $K_{X}$. Thus, as in \cite{BBJ15}, being aware that our definition of $d^{c}$ differs from theirs of a multiplicative factor,
we say that a positive measure $\mu$ on $X$ has \emph{well-defined Ricci curvature} if it corresponds to a \emph{singular metric} on $K_{X}$, i.e. if locally it is of the form (\ref{eqn:Ricci}) with $f\in L^{1}_{loc}$, and in this case $Ric(\mu):=dd^{c}f$. Observe that if $\mu_{\phi}$ is the adapted measure of Definition \ref{defn:Adap} then $Ric(\mu_{\phi})=\omega+[D]$ where $\omega$ is the curvature form of $\phi$.

Then, letting $\eta$ be a \emph{semi-Kähler} form, i.e. a closed smooth semipositive $(1,1)$-form such that $\eta^{n}>0$, we set, for $u\in PSH(X,\eta)$, $Ric(\eta+dd^{c}u):=Ric\big(MA_{\eta}(u)\big)$ so that $Ric(\eta):=Ric(\eta^{n})$ is the usual Ricci curvature when $\eta$ is actually Kähler.
\begin{defn}
Let $D$ be a $\mathbbm{Q}$-divisor and $\eta$ a (semi-)Kähler form. A \emph{$D$-log (semi-)Kähler Einstein metric} on $X$ in the cohomology class $\{\eta\}$ is a positive current $\eta_{u}:=\eta+dd^{c}u$ with well-defined Ricci curvature such that
$$
Ric(\eta_{u})-[D]=\lambda \eta_{u}
$$
for $\lambda\in\mathbbm{R}$ where $[D]$ is the current of integration along the divisor $D$. Furthermore, when $\eta$ is Kähler, if $\eta_{u}$ is a $D$-log KE metric and $u\in\mathcal{E}^{1}(X,\eta,\psi)$ for $\psi\in\mathcal{M}$, then we say that $\eta_{u}$ is $(D,[\psi])$\emph{-log KE metric}.
\end{defn}
Note that a $(D,[0])$-log KE metric is called $[D]$-twisted KE in \cite{BBJ15}, and that the abuse of language is due to the fact that $(D,[\psi])$-log KE metrics define (classes of) singular $D$-log KE metrics. \\
When $D=0$ one recovers the definition of \emph{(semi-)KE metrics} (which coincides with the usual definition of Kähler-Einstein metrics in the Kähler case if the smoothness holds).

It is immediate to see that there is the topological obstruction 
\begin{equation}
\label{eqn:TO}
c_{1}(X)-\{[D]\}=\lambda\{\eta\}
\end{equation}
to the existence of $D$-log semi-KE metrics. However under the assumption (\ref{eqn:TO}), we recall the following pluripotential description of $D$-log semi-KE metrics.
\begin{prop}
\label{lem:BBJ}
Let $D$ be a $\mathbbm{Q}$-divisor such that $(\ref{eqn:TO})$ holds for $\lambda\in\mathbbm{Q}$ and $\eta$ semi-Kähler form. Let also $\phi$ be a metric on $\lambda \{\eta\}$ with curvature $\lambda \eta$, and let $u\in PSH(X,\eta)$. Then $\eta_{u}$ is a $D$-log semi-KE metric if and only if
\begin{equation}
\label{eqn:MARIC}
MA_{\eta}(u)=e^{-\lambda u+C}\mu_{\phi} 
\end{equation}
for a constant $C\in\mathbbm{R}$ where $\mu_{\phi}$ is the adapted measure associated to $\phi$.
\end{prop}
\begin{proof}
The proof is similar to that of \cite[Lemma 2.2]{BBJ15}, but we write the details as a courtesy to the reader.\\
If $u\in PSH(X,\omega)$ solves (\ref{eqn:MARIC}) then $\eta_{u}$ has well-defined Ricci curvature and
$$
Ric(\eta_{u})=\lambda dd^{c}u + Ric(\mu_{\phi})=\lambda dd^{c}u+ \lambda \eta+ [D]=\lambda \eta_{u}+[D].
$$
Conversely, assume that $\eta_{u}$ has well-defined Ricci curvature and that $Ric(\eta_{u})-[D]=\lambda \eta_{u}$. Letting $D=\sum_{j=1}^{N}a_{j}D_{j}$ for $D_{j}$ prime divisors, $\{s_{j}\}_{j=1}^{N}$ holomorphic sections cutting the divisors $\{D_{j}\}_{j=1}^{N}$ and letting $\{\phi_{j}\}_{j=1}^{N}$ metrics on the associated line bundles, we obtain locally on $X\setminus \mathrm{Supp}(D)$
$$
\mu_{\phi}=e^{-2\sum_{j=1}^{N}a_{j}\log |s_{j}|_{\phi_{j}}}e^{-\tilde{\phi}}i^{n^{2}}\Omega\wedge \bar{\Omega}
$$
where $\tilde{\phi}:=\phi+\sum_{j=1}^{N}a_{j}\phi_{j}$ is a metric on $-K_{X}$. In particular, $\mu_{\phi}=e^{-2\sum_{j=1}^{N}a_{j}\log|s_{j}|_{\phi_{j}}}dV$ for a smooth volume form $dV$ such that $Ric(dV)=\lambda\eta+[D]-2\sum_{j=1}^{N}a_{j}dd^{c}\log |s_{j}|_{\phi_{j}}=\lambda \eta +\sum_{j=1}^{n}a_{j}dd^{c}\phi_{j}$. Therefore, letting $f\in L^{1}$ such that $MA_{\eta}(u)=e^{-f}dV$, it follows that
$$
Ric(\eta_{u})=dd^{c}f+Ric(dV)=dd^{c}f+\lambda \eta+[D]-2\sum_{j=1}^{N}a_{j}dd^{c}\log|s_{j}|_{\phi_{j}},
$$
which in turn implies that $f-\lambda u-2\sum_{j=1}^{N}a_{j}\log|s_{j}|_{\phi_{j}}$ is pluriharmonic as $Ric(\eta_{u})=\lambda\eta_{u}+[D]$. Hence there exists a constant $C\in\mathbbm{R}$ such that
$$
MA_{\eta}(u)=e^{-\lambda u+C}e^{-2\sum_{j=1}^{n}a_{j}\log|s_{j}|_{\phi_{j}}}dV=e^{-\lambda u+C}\mu_{\phi},
$$
which concludes the proof.
\end{proof}
\begin{rem}
\label{rem:RLine}
\emph{It is possible to extend the definition of $D$-log (semi-)KE metrics to $\lambda\in \mathbbm{R}$, $D$ $\mathbbm{R}$-divisor thanks to the pluripotential description of Proposition \ref{lem:BBJ}. Indeed in this case $\lambda \eta$ can be thought as the curvature of a metric $\phi$ on a $\mathbbm{R}$-line bundle, i.e. on a formal real combination of line bundles. More precisely if $\{\lambda \eta\}=\{\sum_{k=1}^{m}b_{k}L_{k}\}$ where $b_{k}\in\mathbbm{R}$ and $L_{k}$ line bundles, then there exist metrics $\phi'_{k}$ on $L_{k}$ such that $\phi:= \sum_{k=1}^{m}b_{k}\phi'_{k}$ satisfies $dd^{c}\phi=\lambda\eta$. Next if $D=\sum_{j=1}^{N}a_{j}D_{j}$ for $D_{j}$ prime divisors, we fix $\{s_{j}\}_{j=1}^{N}$ holomorphic sections cutting the divisors $D_{j}$ and metrics $\phi_{j}$ on the associated line bundle. Thus setting $\tilde{\phi}:=\phi+\sum_{j=1}^{N}a_{j}\phi_{j}$ the local volume forms $e^{-\tilde{\phi}}i^{n^{2}}\Omega\wedge \bar{\Omega}$ glue together to give a global volume form $dV$. Set $\mu_{\phi}:=e^{-2\sum_{j=1}^{N}a_{j}\log |s_{j}|_{\phi_{j}}}dV$, where we mean the trivial extension to $0$ of the measure of the right-hand side restricted to $X\setminus \mathrm{Supp}(D)$. We say that $\eta+dd^{c}u$ is a $D$-log (semi-)KE metric if $MA_{\eta}(u)=e^{-\lambda u+C}\mu_{\phi}$ for a constant $C\in\mathbbm{R}$, and if $\eta$ is Kähler we say that $\eta+dd^{c}u$ is a $(D,[\psi])$-log KE metric if we further have $u\in\mathcal{E}^{1}(X,\eta,\psi)$. Note that this definition of $D$-log KE metrics does not depend on the choice done on the metrics. Moreover if $p:Y\to X$ is given by a sequence of blow-ups of smooth centers $\tilde{\mu}_{\phi}=\mu_{p^{*}\phi}$ and $p_{*}\mu_{p^{*}\phi}=\mu_{\phi}$.}
\end{rem}
It is not difficult to check that the adapted measure $\mu_{\phi}$ has finite total mass if and only if $D$ is \emph{klt} (see \cite{KolPairs}), which reads as $a_{j}<1$ if $D=\sum_{j=1}^{N}a_{j}D_{j}$ for prime divisors $D_{j}$ with simple normal crossing. A similar condition holds when one considers $(D,[\psi])$-log KE metrics. Indeed letting $\{s_{j}\}_{j=1}^{N}$, $\{\phi_{j}\}_{j=1}^{N}$ and $dV$ as in proof of Proposition \ref{lem:BBJ}, i.e.
$$
\mu_{\phi}=e^{-2\sum_{j=1}^{N}a_{j}\log|s_{j}|_{\phi_{j}}}dV,
$$
the following necessary condition to the existence of $(D,[\psi])$-log KE metrics in terms of multiplier ideal sheaves can be easily deduced from Proposition \ref{lem:BBJ}.
\begin{cor}
\label{cor:KLT}
Let $\eta$ be a Kähler form such that (\ref{eqn:TO}) holds for $D=\sum_{j}a_{j}D_{j}$ $\mathbbm{R}$-divisor and $\lambda\in\mathbbm{R}$. If $\eta_{u}$ is a $(D,[\psi])$-log KE metric, then
\begin{gather}
\label{eqn:KLT}
\mathcal{I}\Big(\lambda\psi+2\sum_{\{j: a_{j}>0\}}a_{j}\log|s_{j}|_{\phi_{j}}-\sum_{\{j:a_{j}<0\}}\min\big\{\lambda\nu(\psi,D_{j}),-2a_{j}\big\}\log|s_{j}|_{\phi_{j}}\Big)=\mathcal{O}_{X}\, \quad \, \mbox{if}\, \,\,\lambda>0,\\
\label{eqn:KLT2}
\mathcal{I}\Big(2\sum_{\{j: a_{j}>0\}}a_{j}\log|s_{j}|_{\phi_{j}}\Big)=\mathcal{O}_{X}\, \quad \, \mbox{if}\,\,\, \lambda=0,\\
\label{eqn:KLT3}
\mathcal{I}\Big(\sum_{\{j: a_{j}>0\}}\sup\big\{2a_{j}-\lambda\nu(\psi,D_{j}),0\big\}\log|s_{j}|_{\phi_{j}}\Big)=\mathcal{O}_{X}\, \quad \, \mbox{if}\,\,\, \lambda<0,
\end{gather}
where $\nu(\psi,D_{j}):=\inf_{x\in D_{j}} \nu(\psi,x)$ is the Lelong number of $\psi$ along $D_{j}$.
\end{cor}
We will say that $(D,[\psi])$ is \emph{klt} for $D$ $\mathbbm{R}-$divisor such that $c_{1}(X)-\{[D]\}=\lambda \{\eta\}$ when the associated condition based on the sign of $\lambda$ among (\ref{eqn:KLT}), (\ref{eqn:KLT2}), (\ref{eqn:KLT3}) holds. The definition does not depends on the metrics $\phi_{j}$ chosen and it is coherent with the usual definition (see for instance \cite[Proposition 8.2]{KolPairs}).\newline
Note that in the case $\lambda<0$, this extended klt condition can be satisfied for pairs $(X,D)$ that are not klt if the singularities of $\psi$ compensate those of $D$. However, we do not investigate further these situations as they are beyond the purpose of this article.\newline
\subsection{Analytical Singularities.}
In this subsection $\omega$ Kähler and $\psi:=P_{\omega}[\varphi]\in \mathcal{M}^{+}$ for $\varphi\in PSH(X,\omega)$ has \emph{analytical singularities}, i.e. locally $\varphi_{|U}:=g+c\log \big(|f_{1}|^{2}+\cdots +|f_{k}|^{2}\big)$ where $c\in\mathbbm{R}_{\geq 0}$, $g\in C^{\infty}$, and $\{f_{j}\}_{j}^{k}$ are local holomorphic functions. The coherent ideal sheaf $\mathcal{I}$ generated by these functions has integral closure globally defined, hence the singularities of $\varphi$ are formally encoded in $(\mathcal{I},c)$. Indeed, it is well-known that in this case there exists a smooth resolution $p: Y\to X$ given by a sequence of blow-ups of smooth centers such that $p^{*}\mathcal{I}=\mathcal{O}_{Y}(-D)$ for an effective divisor $D$. Moreover the Siu Decomposition (\cite{Siu74}) of $p^{*}(\omega_{\varphi})$ is given by
$$
p^{*}(\omega_{\varphi})=\eta+c[D]
$$
where $\eta$ is a smooth semipositive $(1,1)$-form on $Y$ that becomes semi-Kähler if $\int_{X}\eta^{n}>0$. Recall also that for semi-Kähler forms $\eta$ the sets $\mathcal{E}(Y,\eta)$, $\mathcal{E}^{1}(Y,\eta)$ are defined as in the Kähler case (see \cite{BEGZ10}).
\begin{lem}
\label{lem:Isom}
In the setting just described $\int_{X}\eta^{n}=\int_{X}MA_{\omega}(\varphi)$ and there is a bijective map $F:PSH(X,\omega,\psi)\to PSH(X,\eta)$ such that $F\big(\mathcal{E}(X,\omega,\psi)\big))=\mathcal{E}(Y,\eta)$ and $F\big(\mathcal{E}^{1}(X,\omega,\psi)\big)=\mathcal{E}^{1}(Y,\eta)$.
\end{lem}
Recall that $PSH(X,\omega,\psi)=\{u\in PSH(X,\omega)\,:\, u\preccurlyeq \psi\}$.
\begin{proof}
By \cite[Remark 4.6]{RWN14} $\psi-\varphi$ is globally bounded, so any $u\in PSH(X,\omega,\psi)$ satisfies $u\preccurlyeq \varphi$, which implies that $p^{*}(\omega_{u})-c[D]$ is a closed and positive current on $Y$ with cohomology class $\{\eta\}$. Therefore there exists a unique $\tilde{u}\in PSH(Y,\eta)$ such that $\sup_{Y}\tilde{u}=\sup_{X}(u-\varphi)$ and
$$
p^{*}(\omega_{u})=\eta_{\tilde{u}}+c[D].
$$
Thus we define $F:PSH(X,\omega,\psi)\to PSH(Y,\eta)$ as $F(u):=\tilde{u}$, noting that it is a bijection (see for instance \cite[Proposition 1.2.7.(ii)]{BouTh}). It is also easy to check that $\tilde{u}-(u-\varphi)\circ p $ is pluriharmonic on $Y$, which yields $F(u)=\tilde{u}=(u-\varphi)\circ p$.\\
Next, since $p$ is an isomorphism over $Y\setminus p^{-1}V(\mathcal{I})$ and $[D]$ has support in a pluripolar set, it is not difficult to check that
\begin{equation}
\label{eqn:TP}
p_{*}MA_{\eta}(\tilde{u})=MA_{\omega}(u)
\end{equation}
using the definition of non-pluripolar product. Thus (\ref{eqn:TP}) immediately gives $F\big(\mathcal{E}(X,\omega,\psi)\big)=\mathcal{E}(Y,\eta)$. Hence to conclude the proof it is enough to observe that the equalities
$$
\int_{Y}\tilde{u}MA_{\eta}(\tilde{u})=\int_{Y}p^{*}p_{*}\big((u-\varphi)\circ p\, MA_{\eta}(\tilde{u})\big)=\int_{Y}p^{*}\big((u-\varphi)MA_{\omega}(u)\big)=\int_{X}(u-\psi)MA_{\omega}(u)+\int_{X}(\varphi-\psi)MA_{\omega}(u)
$$
imply $F\big(\mathcal{E}^{1}(X,\omega,\psi)\big)=\mathcal{E}^{1}(Y,\eta)$ as $|\varphi-\psi|\leq C$ and the energies $E_{\psi}(u), E(\tilde{u})$ respectively on $(X,\omega), (Y,\eta)$ are comparable respectively with $\int_{X}(u-\psi)MA_{\omega}(u), \int_{X}\tilde{u}MA_{\eta}(\tilde{u})$ (see \cite[Theorem 4.10]{DDNL17b}).
\end{proof}
For completeness we also prove that in this setting the metric space $\big(\mathcal{E}^{1}(X,\omega,\psi),d\big)$ is isometric to the metric space $\big(\mathcal{E}^{1}(Y,\eta), d\big)$ studied in \cite{DDNL17b} where
$$
d(u,v)=E(u)+E(v)-2E\big(P_{\eta}(u,v)\big)
$$ 
for any $u,v\in\mathcal{E}^{1}(Y,\eta)$
recalling that $P_{\eta}(\cdot,\cdot), E(\cdot)$ are defined similarly to the Kähler case, i.e. for instance $E(u)=\frac{1}{n+1}\sum_{j=0}^{n}\int_{X}u(\eta+dd^{c}u)^{j}\wedge \eta^{n-j}$ if $u$ has minimal singularities and $E(u)=\lim_{k\to \infty}E\big(\max(u,-k)\big)$ otherwise.
\begin{prop}
\label{prop:Isom}
The metric space $\big(\mathcal{E}^{1}(X,\omega,\psi),d\big)$ is isometric to $\big(\mathcal{E}^{1}(Y,\eta),d\big)$ through the map of Lemma \ref{lem:Isom}
\end{prop}
\begin{proof}
With the same notation of Lemma \ref{lem:Isom} we have $\tilde{u}:=F(u)=(u-\varphi)\circ p$ for any $u\in\mathcal{E}^{1}(X,\omega,\psi)$. Moreover similarly as in the proof on Lemma \ref{lem:Isom} we can show that $p_{*}\big(\eta_{\tilde{u}_{1}}^{k}\wedge \eta_{\tilde{u}_{2}}^{n-k}\big)=\omega_{u_{1}}^{k}\wedge \omega_{u_{2}}^{n-k}$ for any $k=0,\dots, n$, and that these equalities lead to $E(\tilde{u})=E_{\psi}(u)-E_{\psi}(\varphi)$ for any $u\in\mathcal{E}^{1}(X,\omega,\psi)$. Hence to conclude the proof it is enough to check that $F\big(P_{\omega}(u_{1},u_{2})\big)=P_{\omega}(\tilde{u}_{1},\tilde{u}_{2})$. By construction we easily have $\tilde{u}_{1}\leq \tilde{u}_{2}$ if and only if $u_{1}\leq u_{2}$. Therefore $F\big(P_{\omega}(u_{1},u_{2})\big)\leq P_{\omega}(	\tilde{u}_{1},\tilde{u}_{2})$ follows from $P_{\omega}(u_{1},u_{2})\leq u_{1},u_{2}$, while letting $\phi\in\mathcal{E}^{1}(X,\omega,\psi)$ such that $\tilde{\phi}=P_{\omega}(\tilde{u}_{1},\tilde{u}_{2})$ we have $\phi\leq u_{1},u_{2}$, i.e. $\phi\leq P_{\omega}(u_{1},u_{2})$, which gives the reverse inequality by composing with $F$.
\end{proof}
We can now relate the $(D,[\psi])$-log KE metrics on $X$ with the $D'$-log semi-KE metrics on $Y$. More precisely, let $D$ be a $\mathbbm{R}$-divisor on $X$ such that
$$
c_{1}(X)-\{[D]\}=\lambda \{\omega\}
$$
for $\lambda\in\mathbbm{R}$ and $\omega$ Kähler form. Let $\psi\in\mathcal{M}^{+}$ given as $P_{\omega}[\varphi]$ for a function $\varphi\in PSH(X,\omega)$ with analytic singularities encoded in $(\mathcal{I},c)$, and let $p: Y\to X$ be a smooth resolution of $\mathcal{I}$. Then $p^{*}\mathcal{I}=\mathcal{O}_{Y}(-D_{1})$ for an effective divisor $D_{1}$ and $p^{*}(K_{X}+D)=K_{Y}+D_{2}$ for a $\mathbbm{R}$-divisor $D_{2}$. We denote with $\eta$ the semi-Kähler part of the Siu Decomposition $p^{*}(\omega_{\varphi})=\eta+c[D_{1}]$.
\begin{prop}
\label{prop:Correspondence}
In the setting described above, there is a bijection beetwen the set of all $(D,[\psi])$-log KE metrics on $X$ in the cohomology class $\{\omega\}$ and the set of all $D'$-log semi-KE metrics on $Y$ in the cohomology class $\{\eta\}$ where $D':=\lambda c[D_{1}]+[D_{2}]$. More precisely letting $\phi_{\omega}$ and $\phi_{\eta}$ be metrics respectively on the $\mathbbm{R}$-line bundles $-(K_{X}+D), -(K_{Y}+D_{2}+\lambda c D_{1})$ with curvatures $\lambda \omega$ and $\lambda \eta$, a function $u\in\mathcal{E}^{1}(X,\omega,\psi)$ solves $MA_{\omega}(u)=e^{-\lambda u}\mu_{\phi_{\omega}}$ if and only if $\tilde{u}=(u-\varphi)\circ p\in\mathcal{E}^{1}(Y,\eta)$ solves $MA_{\eta}(\tilde{u})=e^{-\lambda\tilde{u}}\mu_{\phi_{\eta}}$.
\end{prop}
\begin{proof}
Let $\phi_{\omega},\phi_{\eta}$ as in the statement. Set also $\phi:=p^{*}\phi_{\omega}-\phi_{\eta}$ metric on $\lambda c D_{1}$ with curvature $\theta:=dd^{c}\phi$. Then for $r_{1}=\frac{1}{\lambda c}\in\mathbbm{R}_{>0}$, $r_{1}\lambda cD_{1}=D_{1}$ is an effective divisor and there exists a holomorphic section $s_{1}$ on the associate line bundle such that $r_{1}\theta+dd^{c}\log|s_{1}|^{2}_{r_{1}\phi}=r_{1}\lambda c[D_{1}]$. Thus, since by construction $\lambda\eta+\theta=p^{*}\lambda \omega$, it follows that
$$
dd^{c}\frac{2}{r_{1}}\log|s_{1}|_{r_{1}\phi}=dd^{c}\lambda \varphi\circ p,
$$
i.e. $\lambda\varphi\circ p= \frac{2}{r_{1}}\log |s_{1}|_{r_{1}\phi} +C$ for a constant $C\in\mathbbm{R}$ that, without loss of generality, we may suppose to be $0$. Therefore the lift of the measure $e^{-\lambda u}\mu_{\phi_{\omega}}=e^{-\lambda(u-\varphi)}e^{-\lambda\varphi}\mu_{\phi_{\omega}}$ becomes
$$
e^{-\lambda \tilde{u}-\frac{2}{r_{1}}\log|s_{1}|_{r_{1}\phi}}\mu_{p^{*}\phi_{\omega}}
$$
where $\tilde{u}=(u-\varphi)\circ p$. Next for $\{a_{j}\}_{j=1}^{N_{1}},\{b_{j}\}_{j=1}^{N_{2}}\subset\mathbbm{R}_{>0}$ and prime divisors $\{D_{2,+,j}\}_{j=1}^{N_{1}}$, $\{D_{2,-,j}\}_{j=1}^{N_{2}}$, we have $D_{2}=\sum_{j=1}^{N_{1}}a_{j}D_{2,+,j}-\sum_{j=1}^{N_{2}}b_{j}D_{2,-,j}$ as the difference of two effective $\mathbbm{R}$-divisors. 
Thus locally on $Y\setminus \big(\mbox{Supp}(D_{1})\cup \mbox{Supp}(D_{2})\big)$ by definition there exists $\Omega$ nowhere zero local holomorphic section of $K_{Y}$ such that
$$
\mu_{p^{*}\phi_{\omega}}=e^{-(p^{*}\phi_{\omega}+2\sum_{j=1}^{N_{1}}a_{j}\log|s_{2,+,j}|-2\sum_{j=1}^{N}b_{j}\log|s_{2,-,j}|)}i^{n^{2}}\Omega\wedge \bar{\Omega}
$$
where $\{s_{2,+,j}\}_{j=1}^{N_{1}}, \{s_{2,-,j}\}_{j=1}^{N_{2}}$ are holomorphic sections cutting respectively $\{D_{2,+,j}\}_{j=1}^{N_{1}}, \{D_{2,-,j}\}_{j=1}^{N_{2}}$. For simplicity of notations we set $\varphi_{2,+}:=2\sum_{j=1}^{N_{1}}a_{j}\log|s_{2,+,j}|$ and similarly for $\varphi_{2,-}$. Therefore locally on $Y\setminus \big(\mbox{Supp}(D_{1})\cup \mbox{Supp}(D_{2})\big)$
\begin{multline*}
e^{-\frac{2}{r_{1}}\log|s_{1}|_{r_{1}\phi}}\mu_{p^{*}\phi_{\omega}}=e^{-\big(\phi+\frac{2}{r_{1}}\log|s_{1}|_{r_{1}\phi}+\phi_{\eta}+\varphi_{2,+}-\varphi_{2,-}\big)}i^{n^{2}}\Omega\wedge\bar{\Omega}
=e^{-\big(\phi_{\eta}+\frac{2}{r_{1}}\log|s_{1}|+\varphi_{2,+}-\varphi_{2,-}\big)}i^{n^{2}}\Omega\wedge \bar{\Omega}=\mu_{\phi_{\eta}}.
\end{multline*}
In conclusion, for any $u\in\mathcal{E}^{1}(X,\omega,\psi)$, the measures $e^{-\lambda u}\mu_{\phi_{\omega}}$ and $e^{-\lambda \tilde{u}}\mu_{\phi_{\eta}}$ are related by lifting and by push-forward through $p_{*}$. Moreover, the same correspondence holds for holds for $MA_{\omega}(u)$ and $MA_{\eta}(\tilde{u})$ as seen during the proof of Lemma \ref{lem:Isom}. The Proposition follows.
\end{proof}
We can prove the following regularity result on $(D,[\psi])$-log semi-KE metrics in this situation. It represents the first part of Theorem \ref{thmE}.
\begin{thm}
\label{thm:Regularity}
Let $\omega_{u}$ be a $(D,[\psi])$-log KE metric where $D$ is a $\mathbbm{R}$-divisor and $\psi=P_{\omega}[\varphi]\in\mathcal{M}^{+}$ for $\varphi$ with analytic singularities formally encoded in $(\mathcal{I},c)$. Then $u\in C^{\infty}\big(X\setminus A\big)$ where $A=V(\mathcal{I})\cup \mbox{Supp}(D)$.
\end{thm}
\begin{proof}
By Proposition \ref{prop:Correspondence} and with the same notations, $\tilde{u}:=(u-\varphi)\circ p$ is a solution of
\begin{equation*}
\begin{cases}
MA_{\eta}(\tilde{u})=e^{-\lambda\tilde{u}}\mu_{\phi_{\eta}}\\
\tilde{u}\in\mathcal{E}^{1}(Y,\eta)
\end{cases}
\end{equation*}
where $\eta$ is semi-Kähler form. 
Moreover, writing $\mu_{\phi_{\eta}}=e^{v_{1}-v_{2}}dV$ where $v_{1},v_{2}\in PSH(Y,\omega')$ for $\omega'$ Kähler form and $dV$ volume form on $Y$, by the Monge-Ampère equation and the resolution of the openness conjecture (\cite[Theorem 1.1]{GZ14}) we immediately obtain $e^{-\lambda \tilde{u}+v_{1}-v_{2}}\in L^{p}$ for $p>1$ (see also Corollary \ref{cor:KLT}).\\
Next, the proof is standard. Indeed by \cite[Theorem C]{EGZ10} we get that $\tilde{u}$ is bounded on $X$ and continuous on $\mbox{Amp}(\{\eta\})$ (see also \cite{Kol98}), where the latter is the \emph{ample locus} of $\eta$, i.e. the complementary of the non-Kähler locus (\cite[Definition 3.16]{Bou02}). Then, if $\lambda>0$, fix $C>0$ big enough such that $\sup_{X}v_{1}\leq C$, $C\omega'+dd^{c}v_{1}\geq 0, C\omega'+dd^{c}(v_{1}+\lambda \tilde{u})\geq 0$ and $||e^{-\lambda\tilde{u}-v_{2}}||_{L^{p}}\leq C$. Thus by \cite[Theorem 10.1]{BBEGZ16} for any relatively compact open set $U\Subset \mbox{Amp}(\{\eta\})$ there exists $A>0$ depending on $C, \eta, p, U$ such that
\begin{equation}
\label{eqn:Bound1}
0\leq \eta+dd^{c}\tilde{u}\leq A e^{-\lambda\tilde{u}-v_{2}}\omega'.    
\end{equation}
Similarly, if $\lambda\leq 0$, letting $C>0$ big enough such that $\sup_{X}(v_{1}-\lambda \tilde{u})\leq C$, $C\omega'+dd^{c}(v_{1}-\lambda \tilde{u})\geq 0, C\omega'+dd^{c}v_{2}\geq 0$ and $||e^{-v_{2}}||_{L^{p}}\leq C$, we obtain
\begin{equation}
\label{eqn:Bound2}
0\leq \eta+dd^{c}\tilde{u}\leq Ae^{-v_{2}}\omega'
\end{equation}
for any relatively compact open set $U\Subset \mbox{Amp}(\{\eta\})$.\\
Moreover by construction $v_{1},v_{2}$ are smooth outside the union of the supports of the divisors $D_{1}$, $D_{2}$ (with the notations used in Proposition \ref{prop:Correspondence}).
So, as $\tilde{u}$ is globally bounded, by (\ref{eqn:Bound1}), (\ref{eqn:Bound2}) it immediately follows that $\Delta_{\omega'}\tilde{u}$ is locally bounded over $\mbox{Amp}(\{\eta\})\cap \Big(Y\setminus \big(\mbox{Supp}(D_{1})\cup \mbox{Supp}(D_{2})\big)\Big)$. Thus, the Evans-Krylov Theorem and a classical bootstrap argument imply that $\tilde{u}$ is smooth over $\mbox{Amp}(\{\eta\})\cap\Big(Y\setminus \big(\mbox{Supp}(D_{1})\cup \mbox{Supp}(D_{2})\big)\Big)$. Moreover, the ample locus is a not-empty Zariski open set ($\{\eta\}$ is big, see \cite[Theorem 3.17]{Bou02}) and it includes $Y\setminus \big(\mbox{Supp}(D_{1})\cup \mbox{Supp}(D_{2})\big)$ as $\{\omega\}$ is Kähler and the support of the exceptional locus of $p:Y\to X$ is contained in the union of the supports of $D_{1}, D_{2}$. Hence, as $\tilde{u}=(u-\varphi)\circ p$, we get that $u\in\mathcal{C}^{\infty}(X\setminus B)$ for $B:=p_{*}\big(\mbox{Supp}(D_{1})\cup \mbox{Supp}(D_{2})\big)\subset V(\mathcal{I})\cup \mbox{Supp}(D)=A$, which concludes the proof. 
\end{proof}
\subsection{Theorem \ref{thmE}.}
In the subsection we conclude the proof of Theorem \ref{thmE}.\\

As shown in the previous subsection if $\psi\in\mathcal{M}^{+}$ has analytic singularities type, i.e. $\psi=P_{\omega}[\varphi]$ for $\varphi$ with analytic singularities formally encoded in $(\mathcal{I},c)$ where $\mathcal{I}$ is a integrally closed coherent ideal sheaf and $c\in\mathbbm{R}_{>0}$, then taking a resolution $p:Y\to X$ of $\mathcal{I}$ there exists a semi-Kähler form $\eta$ on $Y$ such that $p^{*}(\omega_{\varphi})=\eta+c[D]$ where $p^{*}\mathcal{I}=\mathcal{O}_{X}(-D)$ and $D$ is an effective divisor. Thus, we first set $\mathcal{M}_{an}^{+}:=\{\psi\in\mathcal{M}^{+}\, \mbox{with analytic singularities type}\}$ and we fix for any $\psi\in\mathcal{M}_{an}^{+}$ an element $\varphi$ with analytic singularities such that $\sup_{X}\varphi=0$ and $\psi=P_{\omega}[\varphi]$ (i.e. $\psi-\varphi$ globally bounded). Then setting $\mathcal{K}_{(X,\omega)}^{\,tot}:=\{(Y,\eta)\, : \, \omega-p_{*}\eta=[D] \, \mbox{for an effective}\, \mathbbm{R}\mbox{-divisor}\, D \, \mbox{where}\, \eta \, \mbox{is semi-Kähler} \, \mbox{and}\, p:Y\to X \, \mbox{is given by a sequence of blow-ups}\}$, the construction described above yields a natural map
$$
\Phi: \mathcal{M}_{an}^{+}\longrightarrow \mathcal{K}_{(X,\omega)}^{\,tot}/\sim
$$
where $(Y,\eta)\sim (Y',\eta')$ on $\mathcal{K}_{(X,\omega)}^{\,tot}$ if there exists $(Z,\tilde{\eta})\in \mathcal{K}_{(X,\omega)}^{\,tot}$ such that $Z$ dominates $Y,Y'$ through morphism $q:Z\to Y$, $q':Z\to Y'$ and $\tilde{\eta}=q^{*}\eta=q'^{*}\eta'$. Note that for a different choice of the elements $\varphi$ with analytic singularities, the forms $\eta$ in the representatives in $\mathcal{K}_{(X,\omega)}^{\,tot}$ may change but their cohomology classes $\{\eta\}$ would remain the same.\\
We also claim that $\Phi$ is injective. Indeed letting $\psi_{1},\psi_{2}\in\mathcal{M}_{an}^{+}$ and letting $(Y,\eta_{1}), (Y,\eta_{2})$ be representatives on the same manifold $Y$ (taking a common resolution), if $\Phi(\psi_{1})=\Phi(\psi_{2})$ then $\eta_{1}=\eta_{2}$. Thus, denoting with $\varphi_{1},\varphi_{2}$ the associated and fixed functions with analytic singularities, the equality $\eta_{1}=\eta_{2}$ and cohomological reasons easily imply that $(\varphi_{1}-\varphi_{2})\circ p$ is pluriharmonic. Hence $\varphi_{1}=\varphi_{2}+C$, which clearly gives $\psi_{1}=\psi_{2}$.\\
We can now define
$$
\mathcal{K}_{(X,\omega)}:=\mathrm{Im}(\Phi).
$$
It is worth to underline that in any cohomology class $\{\mu_{N}^{*}\omega-a_{1}[E_{1}]-a_{2}[E_{2}]+\dots-a_{N}[E_{N}])\}$ given by a \emph{small pertubation} for $\mu_{N}:Y\to X$ blow-up of $X$ at $N$ distinct points, $E_{i}$ exceptional divisors and $a_{i}>0$ small enough, there exists a smooth semi-Kähler form $\eta$ such that $[(Y,\eta)]\in \mathcal{K}_{(X,\omega)}$.\\
As an immediate consequence of the definition, the set $\mathcal{K}_{(X,\omega)}$ inherits a partial order. Indeed, for any $\alpha,\alpha'\in \mathcal{K}_{(X,\omega)}$ with associated model type envelopes $\psi,\psi'\in\mathcal{M}^{+}_{an}$, we will say that $\alpha$ is \emph{smaller} (resp. \emph{bigger}) than $\alpha'$ if $\psi\preccurlyeq \psi'$ (resp. $\psi\succcurlyeq \psi'$). Note that if $\alpha$ is smaller than $\alpha'$ then, taking representatives $(Z,\tilde{\eta}), (Z,\tilde{\eta'})$ on the same compact Kähler manifold $Z$, it follows that $\tilde{\eta'}-\tilde{\eta}=[F]$ for an effective $\mathbbm{R}$-divisor $F$. The volume $\mbox{Vol}(\alpha)$ is also well-defined for $\alpha\in\mathcal{K}_{(X,\omega)}$ as $\int_{Y}\eta^{n}=\int_{Y'}\eta'^{n}$ for any $(Y,\eta)\sim (Y',\eta')$, and in particular $\mbox{Vol}(\alpha)=V_{\psi}$ if $\Phi(\psi)=\alpha$ (see also Lemma \ref{lem:Isom}).\\
Moreover, the notion of log-KE metrics descend to the classes in $\mathcal{K}_{(X,\omega)}$ thanks to Proposition \ref{prop:Correspondence}. Indeed two log-KE metrics $\eta+dd^{c}\tilde{u}$, $\eta'+dd^{c}\tilde{u}'$, respectively on $(Y,\eta),(Y',\eta')$ representatives of the same class in $\mathcal{K}_{(X,\omega)}$, can be identified if $\tilde{u}=(u-\varphi)\circ p, \tilde{u}'=(u-\varphi)\circ p'+C$ for the same function $u\in\mathcal{E}^{1}(X,\omega,\psi)$ and $C\in\mathbbm{R}$. Thus a \emph{log-KE metric} in $\mathcal{K}_{(X,\omega)}$ is a family of log-KE metrics.\newline
We can then define a \emph{strong convergence} of log-KE metrics for totally ordered sequences in $\mathcal{K}_{(X,\omega)}$ using a suitable normalization for the associated quasi-plurisubharmonic functions. Namely, if $\lambda=0$, for any log-KE metric $\eta+dd^{c}\tilde{u}$ on $(Y,\eta)$, representative of a log-KE metric in $\mathcal{K}_{(X,\omega)}$, the function $\tilde{u}$ will be normalized so that the corresponding $\omega$-psh function $u$ through Lemma \ref{lem:Isom} satisfies $\sup_{X}u=0$. Similarly, if $\lambda\neq 0$, we will normalize $\tilde{u}$ so that the associated $u\in PSH(X,\omega)$ satisfies $MA_{\omega}(u)=e^{-\lambda u}\mu_{\phi_{\omega}}$ where $\phi_{\omega}$ is a fixed metric on $-(K_{X}+D)$ with curvature $\lambda \omega$ (see again Proposition \ref{prop:Correspondence}). In conclusion, given a totally ordered sequence $\{\alpha_{k}\}_{k\in\mathbbm{N}}\subset \mathcal{K}_{(X,\omega)}$ converging to $\alpha\in\mathcal{K}_{(X,\omega)}$, we will say that a sequence of log-KE metrics in $\alpha_{k}$, i.e. a sequence of families of log-KE metrics $\eta_{k}+dd^{c}\tilde{u}_{k}$, \emph{converges strongly} to a log-KE metric in $\alpha$ if $u_{k}\to u$ strongly for the associated functions in $PSH(X,\omega)$. In particular, when there exists a common compact Kähler manifold $Z$ such that the log-KE metrics in $\alpha_{k}, \alpha$ have representatives $\theta_{k}+dd^{c}u_{k}, \theta+dd^{c}v$ the strong convergence implies that $\theta_{k}+dd^{c}v_{k}$ converges weakly to $\theta+dd^{c}v$.\\

We can now prove the second part of Theorem \ref{thmE}.
\begin{thm}
\label{thm:E1}
Let $\omega$ be a Kähler form and let $D$ be a klt $\mathbbm{R}$-divisor such that $c_{1}(X)-\{[D]\}=\lambda \{\omega\}$ holds for $\lambda\leq 0$. Then any class in $\mathcal{K}_{(X,\omega)}$ admits a unique log-KE metric. Furthermore, these log-KE metrics are stable with respect to the strong convergence, i.e. if $\{\alpha_{k}\}_{k\in\mathbbm{N}}\subset \mathcal{K}_{(X,\omega)}$ is a totally ordered sequence converging to $\alpha\in\mathcal{K}_{(X,\omega)}$, then the sequence of log-KE metrics in $\alpha_{k}$ converges strongly to the log-KE metric in $\alpha$.
\end{thm}
\begin{proof}
By Proposition \ref{prop:Correspondence} and by definition, finding a log-KE metric on $\alpha\in\mathcal{K}_{(X,\omega)}$ is equivalent to solve
\begin{equation}
\begin{cases}
MA_{\omega}(u)=e^{-\lambda u}\mu_{\phi_{\omega}}\\
u\in\mathcal{E}^{1}(X,\omega,\psi),
\end{cases}
\end{equation}
where $\psi\in\mathcal{M}^{+}$ is the model type envelope with analytic singularities associated to $\alpha$. Moreover, as $D$ is klt, by the resolution of the openness conjecture (\cite[Theorem 1.1]{GZ14}) it follows that $\mu_{\phi_{\omega}}=fdV$ for $f\in L^{p}$ for $p>1$. Therefore Theorems \ref{thmA}, \ref{thmB} conclude the proof.
\end{proof}
Next it remains to treat the case $\lambda>0$.\\
We first note that in the case of $(D,[\psi])$-log KE metrics the density $f_{D}\in {L^{1}}\setminus\{0\}$ of the corresponding Monge-Ampère equation $MA_{\omega}(u)=e^{-\lambda u}f_{D}\omega^{n}$ is given as
\begin{equation}
\label{eqn:KLTDensity}
f_{D}=e^{-\sum_{j=1}^{N}a_{j}\log|s_{j}|^{2}_{\phi_{j}}+g}
\end{equation}
where $g$ is a smooth function, and as usual we fixed $\{s_{j}\}_{j=1}^{N}$ holomorphic sections cutting the prime divisors $D_{j}$ and metrics $\phi_{j}$ on the associated line bundle where $D=\sum_{j=1}^{N}a_{j}D_{j}$. \\
We then observe that in Theorem \ref{thmC} the assumption $c(\psi)>\frac{\lambda p}{p-1}$ was used two times. In the second part of the proof, to have $e^{-\lambda\psi}\in L^{1+\delta}(\mu)$ for $\delta>0$, condition that in the study of log-KE metrics immediately follows if $(D,[\psi])$ is klt. Moreover, $c(\psi)>\frac{\lambda p}{p-1}$ was used in the first part of the proof of Theorem \ref{thmC} to prove that the $d$-coercivity of $F_{f_{D},\psi,\lambda}$ implies the existence of a maximizer. Thus, as a consequence of the next result, this hypothesis is not longer necessary in the study of log-KE metrics in $\mathcal{K}_{(X,\omega)}$.
\begin{lem}
\label{lem:Ding}
Let $\omega$ be a Kähler form such that $c_{1}(X)-\{[D]\}=\lambda \{\omega\}$ holds for $\lambda>0$ and $D$ $\mathbbm{R}$-divisor. Let also $\psi\in\mathcal{M}^{+}_{an}$ and assume that $(D,[\psi])$ is klt. Then the $d$-coercivity of $F_{f_{D},\psi,\lambda}$ over $\mathcal{E}^{1}_{norm}(X,\omega,\psi)$ implies the existence of a maximizer of $F_{f_{D},\psi,\lambda}$.
\end{lem}
We recall that the definition of klt for $(D,[\psi])$ is given after Corollary \ref{cor:KLT}.
\begin{proof}
Fix $(Y,\eta)$ representative of $\alpha=\Phi(\psi)\in\mathcal{K}_{(X,\omega)}$, $\varphi\in PSH(X,\omega)$ with analytic singularities such that $\psi-\varphi$ is globally bounded.\newline 
Then, as shown in Proposition \ref{prop:Correspondence} and with the same notations, for any $v\in\mathcal{E}^{1}(X,\omega,\psi)$ by the lift of $e^{-\lambda v}f_{D}\omega^{n}=e^{-\lambda v}\mu_{\phi_{\omega}}$ to $Y$ is $e^{-\lambda \tilde{v}}\mu_{\phi_{\eta}}$ where $\tilde{v}=(v-\varphi)\circ p$. Thus, using also Proposition \ref{prop:Isom}, it follows that
\begin{equation}
\label{eqn:Corr}
F_{f_{D},\psi,\lambda}(v)-E_{\psi}(\varphi)=E(\tilde{v})+\frac{V_{\psi}}{\lambda}\log\int_{X}e^{-\lambda \tilde{v}}\mu_{\phi_{\eta}}=:D_{\eta}(\tilde{v})
\end{equation}
for any $v\in\mathcal{E}^{1}(X,\omega,\psi)$. Observe that, up to rescaling the class $\omega$, since $V_{\psi}=\int_{X}\eta^{n}$, the functional $D_{\eta}$ coincides with the \emph{log-Ding functional} in the class $\{\eta\}$ as described in \cite{BBEGZ16}. Moreover, as a consequence of the $d$-coercivity of $F_{f_{D},\psi,\lambda}$ and of the isometry $\big(\mathcal{E}^{1}(X,\omega,\psi),d\big)\ni u\to \tilde{u}\in\big(\mathcal{E}^{1}(Y,\eta),d\big)$ (Proposition \ref{prop:Isom}), it follows that $D_{\eta}$ is $d$-coercive over $\mathcal{E}^{1}_{norm}(Y,\eta)$. Thus we fix a maximizing sequence $\{\tilde{v}_{k}\}_{k\in\mathbbm{N}}\subset \mathcal{E}^{1}_{norm}(Y,\eta)$ that without loss of generality by the compactness of $\{\tilde{v}\in PSH(Y,\eta)\, : \, \sup_{Y}\tilde{v}=0\}$ we may assume to be weakly convergent to $\tilde{v}\in\mathcal{E}^{1}_{norm}(Y,\eta)$. Writing $\mu_{\phi_{\eta}}=g dV$ where $g\in L^{p}$ for $p>1$ and $dV$ is a smooth volume form, we also fix $a\in \mathbbm{R}$ such that $p>a>1$ and we denote by $q\in (1,+\infty)$ the Sobolev conjugate of $p/a$. Then, using the trivial inequality $|e^{a}-e^{b}|\leq e^{a+b}|a-b|$ for $a,b>0$ and applying twice the Hölder's inequality, we have
\begin{multline}
\label{eqn:LastLine}
\int_{X}|e^{-\lambda\tilde{v}_{k}}-e^{-\lambda\tilde{v}}|d\mu_{\phi_{\eta}}\leq\lambda \int_{X}e^{-\lambda(\tilde{v}_{k}+\tilde{v})}|\tilde{v}_{k}-\tilde{v}|d\mu_{\phi_{\eta}}\leq\lambda ||e^{-\lambda(\tilde{v}_{k}+\tilde{v})}||_{L^{q}}||(\tilde{v}_{k}-\tilde{v})g||_{L^{p/a}}\leq\\
\leq\lambda ||e^{-\lambda(\tilde{v}_{k}+\tilde{v})}||_{L^{q}}||g||_{L^{p}}||\tilde{v}_{k}-\tilde{v}||_{L^{p/(a-1)}}.
\end{multline}
Morevoer, as $\eta$ is semi-Kähler, any element in $\mathcal{E}^{1}(Y,\eta)$ has vanishing Lelong numbers (see \cite[Theorem 1.1]{DDNL17a}, Proposition \ref{prop:Lelong} is enough in the Kähler case). Thus, combining Proposition \ref{prop:Alpha} with Lemma \ref{lem:NNPPb} (see \cite{Zer01} for the general case), the first factor in the right side in (\ref{eqn:LastLine}) is uniformly bounded, and the convergence $e^{-\lambda\tilde{v}_{k}}\to e^{-\lambda\tilde{v}}$ in $L^{1}(\mu_{\phi_{\eta}})$ follows from $\tilde{v}_{k}\to \tilde{v}$ in $L^{p}$. Hence, by the upper semicontinuity of $E(\cdot)$ in $\mathcal{E}^{1}(Y,\eta)$ with respect to the weak topology (\cite[Proposition 2.10]{BEGZ10}) we obtain
$$
\sup_{\mathcal{E}^{1}(Y,\eta)}D_{\eta}=\lim_{k\to \infty}D_{\eta}(\tilde{v}_{k})\leq D_{\eta}(\tilde{v}),
$$
i.e. $\tilde{v}$ is a maximizer of $D_{\eta}$. Finally, the equality (\ref{eqn:Corr}) implies that the function $v\in \mathcal{E}^{1}(X,\omega,\psi)$ associated to $\tilde{v}$ (Lemma \ref{lem:Isom}) is a maximizer of $F_{f_{D},\psi,\lambda}$.
\end{proof}
\begin{rem}
\label{rem:ImpRem}
\emph{As seen during the proof of Lemma \ref{lem:Ding}, the $d$-coercivity of $F_{f_{D},\psi,\lambda}$ over $\mathcal{E}^{1}_{norm}(X,\omega,\psi)$ with respect to coefficients $A>0,B\geq 0$ (i.e. $F_{f_{D},\psi,\lambda}(u)\leq-Ad(\psi,u)+B$ for any $u\in\mathcal{E}^{1}_{norm}(X,\omega,\psi)$) is equivalent to the $d$-coercivity of the log-Ding functional $D_{\eta}$ over $\mathcal{E}^{1}_{norm}(Y,\eta)$ with respect coefficients $A>0,B_{\eta}\geq 0$ for any $(Y,\eta)$ representative of the class $\Phi(\psi)\in \mathcal{K}_{(X,\omega)}$. In particular $F_{f_{D},\psi,\lambda}$ and $D_{\eta}$ have the same \emph{slope} at infinity (i.e. the coefficient $A$ of the $d$-coercivity).}
\end{rem}
The following two results are related to Theorems \ref{thmC}, \ref{thmD}, and they conclude the proof of Theorem \ref{thmE}.
\begin{thm}
\label{thm:E2}
Let $\omega$ be a Kähler form such that $c_{1}(X)-\{[D]\}=\lambda \{\omega\}$ holds for $\lambda>0$ and $D$ $\mathbbm{R}$-divisor. Assume also that $(D,[\psi])$ is klt. If the log-Ding functional associated to a representative $(Y,\eta)$ of $\alpha\in\mathcal{K}_{(X,\omega)}$ is $d$-coercive over $\mathcal{E}^{1}_{norm}(Y,\eta)$, then there exists $A>1$ such that any $\alpha'\in\mathcal{K}_{(X,\omega)}$ bigger than $\alpha$ satisfying $\mbox{Vol}(\alpha')< A\mbox{Vol}(\alpha)$ admits a log-KE metric.
\end{thm}
\begin{proof}
It follows directly from Theorem \ref{thmC} thanks to Lemma \ref{lem:Ding} and Remark \ref{rem:ImpRem}.
\end{proof}
\begin{thm}
\label{thm:E3}
Let $\omega$ be a Kähler form, and $D$ be a $\mathbbm{R}$-divisor such that $c_{1}(X)-\{[D]\}=\lambda \{\omega\}$ holds for $\lambda>0$. Assume that
\begin{itemize}
\item[(i)] $\{\alpha_{k}\}_{k\in\mathbbm{N}}\subset \mathcal{K}_{(X,\omega)}$ is an increasing sequence converging to $\alpha\in\mathcal{K}_{(X,\omega)}$;
\item[(ii)] a sequence of log-KE metrics in $\alpha_{k}$ satisfies $\sup_{X}u_{k}\leq C$ for any $k\in\mathbbm{N}$.
\end{itemize}
Then there exists a subsequence of log-KE metrics in $\alpha_{k_{h}}$ converges strongly to a log-KE metric in $\alpha$.
\end{thm}
\begin{proof}
Denote by $\psi_{k}, \psi\in\mathcal{M}^{+}_{an}$ the model type envelopes associated respectively to $\alpha_{k},\alpha$. The required strong convergence of log-KE metrics in $\alpha_{k}$ to a log-KE metric in $\alpha$ is equivalent (by the aforementioned discussions, and in particular by Proposition \ref{prop:Correspondence}) to prove the strong convergence of $u_{k}$ solutions of
$$
\begin{cases}
MA_{\omega}(u_{k})=e^{-\lambda u_{k}}\mu_{\phi_{\omega}}\\
u_{k}\in \mathcal{E}^{1}(X,\omega,\psi_{k});
\end{cases}
$$
to $u\in\mathcal{E}^{1}(X,\omega,\psi)$, proving also that $u$ solves
$$
\begin{cases}
MA_{\omega}(u)=e^{-\lambda u}\mu_{\phi_{\omega}}\\
u\in\mathcal{E}^{1}(X,\omega,\psi).
\end{cases}
$$
Observe that the adapted measure $\mu_{\phi_{\omega}}$ associated to a fixed metric $\phi_{\omega}$ of $-(K_{X}+D)$, satisfies $\mu_{\phi_{\omega}}=f_{D}\omega^{n}$ (\ref{eqn:KLTDensity}) for $f_{D}\in L^{p}$, $p>1$ because $(X,D)$ is klt as a consequence of the existence of log-KE metrics in $\alpha_{k}$ and of Corollary \ref{cor:KLT}.\newline
Therefore, the result clearly follows from Theorem \ref{thmD} showing that the assumption $c(\psi)>\frac{\lambda p}{p-1}$ is superfluous in this case. In proving Theorem \ref{thmD} this hypothesis has been applied to prove that
\begin{equation}
\label{eqn:BABABA}
\int_{X}e^{-\lambda v_{k}}f_{k}\omega^{n}\to \int_{X}e^{-\lambda v}f\omega^{n}
\end{equation}
assuming that $v_{k}\in\mathcal{E}^{1}(X,\omega,\psi_{k})$ converges weakly to $v\in\mathcal{E}^{1}(X,\omega,\psi)$ and that $f_{k}\to f$ in $L^{p}$ for $p>1$, for $f_{k},f$ densities of the associated complex Monge-Ampère equations (in our case $f_{k}\equiv f_{D}$).\newline
Then we claim that there exist representatives $\eta_{k}+dd^{c}\tilde{u}_{k}$ of the sequence of log-KE metrics in $(ii)$ on the same compact Kähler manifold $Y$. Indeed, as the sequence $\psi_{k}$ is increasing, the associated ideal sheaf $\mathcal{I}_{k}$ is also increasing. Therefore by the Strong Noetherian Property of coherent sheaves there exists $k_{0}\in\mathbbm{N}$ such that $\mathcal{I}_{k}=\mathcal{I}_{k_{0}}$ for any $k\geq k_{0}$, and the claim easily follows.\newline
Next, letting $\omega'$ be a Kähler form on $Y$, by construction the lift of $e^{-\lambda v_{k}}\mu_{\phi_{\omega}}$ is $e^{-\lambda \tilde{v}_{k}}g_{k}\omega'^{n}$ for a function $g_{k}=e^{w_{k}^{+}-w_{k}^{-}}$ such that $w_{k}^{+},w_{k}^{-}\in PSH(Y,C\omega')$ for a constant $C\in \mathbbm{R}$ uniform in $k$ (Proposition \ref{prop:Correspondence}). Similarly, the lift of $e^{-\lambda v}\mu_{\phi_{\omega}}$ is $e^{-\lambda \tilde{v}}g\omega'^{n}$ for a function $g=e^{w^{+}-w^{-}}$ such that $w^{+},w^{-}\in PSH(Y,C\omega')$. Observe that by construction $g_{k}\to g$ in $L^{1}$. Moreover, as the singularities are decreasing and there exist log-KE metrics in $\alpha_{k}$, by the resolution to the strong openness conjecture (\cite[Theorem 1.1]{GZ14}) there exists $p>1$ such that $g_{k},g\in L^{p}$. Thus, Proposition \ref{lem:NNPP} yields $g_{k}\to g$ in $L^{p}$.\newline
In conclusion, (\ref{eqn:BABABA}) is equivalent to
$$
\int_{Y}e^{-\lambda \tilde{v}_{k}}g_{k}\omega'^{n}\to \int_{Y}e^{-\lambda \tilde{v}}g\omega'^{n}
$$
where with obvious notations $v_{k}\in \mathcal{E}^{1}(Y,\eta_{k})$ converges weakly to $v\in\mathcal{E}^{1}(Y,\eta)$ and $g_{k}\to g$ in $L^{p}$ for $p>1$. Thus,
\begin{equation}
\label{eqn:SanPol}
\int_{Y}|e^{-\lambda \tilde{v}_{k}}g_{k}-e^{-\lambda \tilde{v}
}g|\omega'^{n}\leq \int_{Y}e^{-\lambda \tilde{v}_{k}}|g_{k}-g|\omega'^{n}+ \int_{Y}|e^{-\lambda \tilde{v}_{k}}-e^{-\lambda \tilde{v}
}|g\omega'^{n},
\end{equation}
and we claim that both the terms in the right-hand side of (\ref{eqn:SanPol}) converge to $0$ as $k\to \infty$. Indeed,
as $v_{k},v$ have vanishing Lelong numbers (\cite[Theorem 1.1]{DDNL17a}), combining Proposition \ref{prop:Alpha} with Lemma \ref{lem:NNPPb} (see \cite{Zer01} for the general case), the first term is dominated by $C'||g_{k}-g||_{L^{p}}$ applying Hölder's inequality, while the remaining part converges to $0$ exactly as in (\ref{eqn:LastLine}).
\end{proof}
We conclude the paper with the following example which shows that the assumption $(iii)$ in Theorem \ref{thmD} (and the assumption $(ii)$ in Theorem \ref{thm:E3}) is necessary.
\begin{esem}
\label{esem:Necess}
\emph{Let $\omega$ be a Kähler form on a Fano manifold $X$, $\{\omega\}=c_{1}(X)$, and let $D$ be a smooth divisor $\mathbbm{Q}$-linearly equivalent to $-K_{X}$, i.e. $D\in |-rK_{X}|$ for $r\in\mathbbm{N}$. \newline
Letting $\varphi_{D}\in PSH(X, \omega)$ such that $\omega+dd^{c}\varphi_{D}=\frac{1}{r}[D]$ and defining $\psi_{t}:=P_{\omega}[t\varphi_{D}]\in\mathcal{M}^{+}$ for any $t\in[0,1)$, by Proposition \ref{prop:Correspondence} there is a bijection between the set of solutions of
\begin{equation}
\label{eqn:Finally}
\begin{cases}
MA_{\omega}(u_{t})=e^{-u_{t}}\mu_{\phi_{\omega}}\\
u_{t}\in\mathcal{E}^{1}(X,\omega,\psi_{t}),
\end{cases}
\end{equation}
and the set of solutions of
\begin{equation}
\label{eqn:BABBAMIA}
\begin{cases}
MA_{(1-t)\omega}(v_{t})=e^{-v_{t}-\frac{t}{r}\varphi_{D}}\mu_{\phi_{\omega}}\\
v_{t}\in\mathcal{E}^{1}(X,(1-t)\omega).
\end{cases}
\end{equation}
where the correspondence is clearly given by $u_{t}=v_{t}+\frac{t}{r}\varphi_{D}$. Note that (\ref{eqn:BABBAMIA}) produces $\frac{t}{r}D$-log KE metrics in the cohomology class $\{(1-t)\omega\}$. Moreover, setting $w_{t}:=\frac{1}{(1-t)}v_{t}\in PSH(X,\omega)$, it easily follows that (\ref{eqn:BABBAMIA}) can be written as
$$
\begin{cases}
MA_{\omega}(w_{t})=(1-t)^{-n}e^{-(1-t)w_{t}-\frac{t}{r}\varphi_{D}}\mu_{\phi_{\omega}}\\
w_{t}\in\mathcal{E}^{1}(X,\omega),
\end{cases}
$$
which in turn is equivalent to the renowned path
\begin{equation}
\label{eqn:CDSPath}
Ric(\omega_{v_{t}})=(1-t)\omega_{v_{t}}+\frac{t}{r}[D]
\end{equation}
considered in \cite{CDS15}. Thus the set $S:=\{t\in[0,1)\, : \, (\ref{eqn:Finally})\, \mbox{admits a solution}\}$ is not empty (\cite[Theorem 1.5]{Berm13}) and open (by a classical Implicit Function Theorem, see \cite{Aub84}). However, if $X$ does not admit a KE metric (for instance if $X=\mathrm{Bl}_{p}\mathbbm{P}^{2}$) then there exists $t_{0}\in (0,1)$ such that $\liminf_{t\searrow t_{0}}\sup_{X}w_{t}=+\infty$, which clearly implies $\liminf_{t\searrow t_{0}}\sup_{X}u_{t}=+\infty$. In fact, otherwise $S$ would be closed and $S=[0,1)$, which would be clearly a contradiction.}
\end{esem}

{\footnotesize
\bibliographystyle{acm}
\bibliography{main}
}
\end{document}